\documentclass[11pt,a4paper]{article}

\usepackage[top=3cm, bottom=3cm, left=2.8cm, right=2.8cm]{geometry}

\usepackage[T1]{fontenc}
\usepackage{latexsym,amssymb}
\usepackage{amsmath}
\usepackage{amsthm}
\usepackage[mathscr]{eucal}
\usepackage{color}
\usepackage{xcolor}

\usepackage[abbrev]{amsrefs}

\newtheorem{Theorem}{\bf Theorem}[section]
\newtheorem{Lemma}{\bf Lemma}[section]
\newtheorem*{lemma*}{\bf Lemma}
\newtheorem{Proposition}{\bf Proposition}[section]
\newtheorem{Corollary}{\bf Corollary}[section]
\newtheorem{Remark}{\bf Remark}[section]
\newtheorem{Example}{\bf Example}[section]
\newtheorem{Definition}{\bf Definition}[section]

\newenvironment{theorem}{\begin{Theorem}}{\end{Theorem}}
\newenvironment{lemma}{\begin{Lemma}}{\end{Lemma}}
\newenvironment{proposition}{\begin{Proposition}}{\end{Proposition}}

\newenvironment{remark}{\begin{Remark}}{\end{Remark}}

\allowdisplaybreaks[4]

\numberwithin{equation}{section}

\def\XXint#1#2#3{{\setbox0=\hbox{$#1{#2#3}{\int}$}
\vcenter{\hbox{$#2#3$}}\kern-.5\wd0}}

\newcommand{\PaperTitle}{Nondegeneracy and asymptotic behavior of type I blow-up for semilinear heat equations with general nonlinearities}

\title{\textcolor{black}{\PaperTitle}}
\author{Yohei Fujishima and Norisuke Ioku}
\author{
        Yohei Fujishima\\
        {\small Department of Mathematical and Systems Engineering} \\
        {\small Faculty of Engineering, Shizuoka University} \\
        {\small 3-5-1 Johoku, Hamamatsu 432-8561, Japan} \\ 
        {\small fujishima@shizuoka.ac.jp}\\ \\
        Ryoto Ichiya \\ 
        {\small Mathematical Institute, Tohoku University} \\
        {\small Aramaki 6-3, Sendai 980-8578, Japan} \\ 
        {\small Current affiliation: Panasonic Digital Co., Ltd.} \\ 
        {\small ichiya.ryoto.t2@alumni.tohoku.ac.jp}\\ \\
        Norisuke Ioku \\
        {\small Mathematical Institute, Tohoku University} \\
        {\small Aramaki 6-3, Sendai 980-8578, Japan}\\
        {\small ioku@tohoku.ac.jp}
        }
\date{}
\begin{document}
\maketitle

\begin{abstract}
We study the local asymptotic behavior of nonnegative type~I blow-up solutions to semilinear heat equations with nonlinearities that need not be scale invariant.
Using \textcolor{black}{a quasi-similarity transformation adapted to} the associated ordinary differential equation, we establish locally uniform convergence to a constant profile on the parabolic spatial scale near every interior blow-up point.
Under a quantitative assumption on the nonlinearity at infinity, our results cover regularly varying nonlinearities in the Sobolev subcritical and critical ranges and, in one and two space dimensions, rapidly growing nonlinearities including 
double exponentials.
A key ingredient in this analysis is a positive lower bound for the normalized solution at interior blow-up points. In fact, under weaker assumptions on the nonlinearity, we prove this bound in arbitrary space dimension without assuming radial symmetry.
\end{abstract}

\section{Introduction}
We study the local asymptotic behavior of finite-time blow-up solutions to
the semilinear heat equation
\begin{equation}
    \label{eq:1.1}
    \begin{cases}
        \partial_t u = \Delta u + f(u),
        & x\in\Omega, \ t>0,
        \\
        u(x,t) = 0, & x\in \partial\Omega, \ t>0,
        \\
        u(x,0) = u_0(x) & x\in \Omega,
    \end{cases}
\end{equation}
where $N\ge 1$ is an integer, $\Omega$ is a domain in $\mathbb{R}^N$, $u_0$ denotes nonnegative initial data,
and $f\in C^1([0,\infty))$ is a nonnegative function.
Throughout the paper, we assume that $f$ satisfies
\begin{equation}
    \label{eq:1.2}
    F(u) :=
    \int_u^\infty \frac{1}{f(s)}\, ds < \infty
\end{equation}
for all $u>0$.

Under these assumptions, we consider nonnegative solutions of \eqref{eq:1.1} that blow up at a finite
time $T\in(0,\infty)$, in the sense that
\[
    \lim_{t\uparrow T}\|u(\cdot,t)\|_{L^\infty(\Omega)} = \infty.
\]
We call $T$ the blow-up time of $u$.
A blow-up solution $u$ is said to be of \textit{type~I} if
\begin{equation}
    \label{eq:1.3}
    \limsup_{t\uparrow T} \frac{T-t}{F\bigl(\|u(\cdot,t)\|_{L^\infty(\Omega)}\bigr)} < \infty,
\end{equation}
and of \textit{type~II} otherwise.
The condition \eqref{eq:1.3} means that the blow-up rate of $u$ is controlled by that of the corresponding ODE
$z'=f(z)$, while type~II blow-up occurs at a faster rate.
Indeed, an ODE solution $z$ that blows up at time $T$ satisfies $F(z(t))=T-t$.
{\color{black}We adopt the definition of type~I blow-up for general nonlinearities used in the first author's work~\cite{Fujis2012} and in joint work with Kan~\cite{FujishimaKan}.
For the non-exponential nonlinearities considered by Souplet~\cite{Soupl2023}, this definition is equivalent to the blow-up rate condition used there.
}

Besides determining whether blow-up occurs and estimating its rate, a basic problem is to describe the spatial behavior of solutions near blow-up points.
For general accounts of blow-up theory, including blow-up rates and profiles, we refer to the survey by Galaktionov and V\'azquez~\cite{GalaktionovVazquez} and the monograph by Quittner and Souplet~\cite{QuittnerSouplet}.
Here, an interior point $a\in\Omega$ is called a blow-up point if $u(x_j,t_j)\to\infty$ for some sequence $(x_j,t_j)\in\Omega\times(0,T)$ with $(x_j,t_j)\to(a,T)$.
An important question is whether the solution near such a point approaches the ODE profile on the parabolic spatial scale $|x-a|=O(\sqrt{T-t})$.

For the power nonlinearity $f(u)=u^p$ with $p>1$, Giga and Kohn addressed this question in a series of papers. They analyzed asymptotic self-similarity in \cite{GigaKohn1985}, obtained blow-up rate estimates in \cite{GigaKohn1987}, and proved nondegeneracy in \cite{GigaKohn1989}, thereby excluding a vanishing rescaled profile at a blow-up point.
For the exponential nonlinearity $f(u)=e^u$, the corresponding asymptotic result in dimensions $N\le2$ was proved by Liu~\cite{Liu}.
Since
\[
    F(u)=\begin{cases}
        \displaystyle\frac{u^{-(p-1)}}{p-1} & \mbox{if }f(u)=u^p,\\[6pt]
        e^{-u} & \mbox{if }f(u)=e^u,
    \end{cases}
\]
these results can be expressed in terms of the common normalization $(T-t)/F(u)$.
The following proposition states that this quantity converges to its ODE value $1$ near every interior blow-up point under the assumptions below.

\begin{proposition}[{\color{black}\cites{GigaKohn1985,GigaKohn1989,Liu}}]
\label{Proposition:1.1}
{\color{black}Assume that $\Omega$ is a bounded convex domain with smooth boundary.\footnote{\color{black}For $f(u)=u^p$, the convexity assumption can be dropped, as shown by Ishige and Mizoguchi~\cite{IshigM2003}.
The boundedness assumption can also be removed by combining the local asymptotic analysis in~\cite{GigaKohn1985} with the nondegeneracy criterion of Mizoguchi and Souplet~\cite{MizogS2019}.}}
Let $f(u)=u^p$ with $p>1$ or $f(u)=e^u$.
Assume that $(N-2)p<N+2$ in the case $f(u)=u^p$, and that $N\le 2$ in the case $f(u)=e^u$.
Let $u$ be a nonnegative type~{\rm I} blow-up solution of \eqref{eq:1.1} in $\Omega$ with blow-up time $T\in(0,\infty)$.
{\color{black}In the case $f(u)=e^u$, assume in addition that $\partial_t u\ge0$ in $\Omega\times(0,T)$.}
If $a\in\Omega$ is a blow-up point of $u$, then
\begin{equation}
    \notag
    \lim_{t\uparrow T}\frac{T-t}{F\bigl(u(a+y\sqrt{T-t},t)\bigr)}=1
\end{equation}
locally uniformly in $y\in\mathbb{R}^N$.
\end{proposition}

Extensions of this asymptotic description to nonlinearities without exact
scale invariance have been obtained in recent work.
Chabi~\cite{Chabi} proved analogous convergence for nonlinearities
$f(u)=u^pL(u)$, where $p>1$ is Sobolev subcritical and the slowly varying
factor $L$ satisfies an additional differential decay condition
(see also Chabi and Souplet~\cite{ChabiS2025} for refined blow-up profiles).
For exponential growth, 
Chabi~\cite{ChabiExponential}
studied
$f(u)=e^uL(e^u)$, with $L$ slowly varying and satisfying suitable decay
conditions on its derivatives. {\color{black}For nonconstant nonnegative solutions on $\mathbb{R}^N$ or a ball
that are radially nonincreasing and satisfy}
$u_t\ge0$, the analysis yields a sharp profile in the original space-time
variables, the final spatial profile, and an expansion including the
first correction to the ODE profile. Under these assumptions, 
the results
are valid in arbitrary space dimension. When $N\le2$, the assumption of monotonicity in time can be replaced
by a type~I blow-up condition
(see \cite{ChabiExponential}*{Remark~3.3}).
Turning to super-exponential growth, Ichiya~\cite{Ichiya} considered
$f(u)=u^q\exp(u^p)$ with $p>1$ and $q\in\{0\}\cup[1,\infty)$,
and established the normalized asymptotic behavior for radially
nonincreasing type~I blow-up solutions in balls when $N\le2$.
These results motivate a criterion expressed directly in terms of $f$
and $F$ that applies to different growth regimes.

{\color{black}In this paper, we characterize the asymptotic profile of a type I blow-up solution for nonlinearities that need not be scale invariant, without assuming radial symmetry.}
To formulate such a criterion and study the behavior near an interior blow-up point $a\in\Omega$, we use
\textcolor{black}{a quasi-similarity transformation combining similarity variables with the normalization provided by $F$}.
Following the transformation approach of \cite{FujishimaIoku}, we \textcolor{black}{introduce the similarity variables}
\[
    y:=\frac{x-a}{\sqrt{T-t}},\qquad s:=-\log(T-t),
\]
and define \textcolor{black}{the normalized function $v$} by
\begin{equation}
\label{eq:1.4}
    v(y,s):=\frac{T-t}{F\bigl(u(a+y\sqrt{T-t},t)\bigr)}.
\end{equation}
For $f(u)=u^p$ $(p>1)$ and $f(u)=e^u$, we have $f'(u)F(u)\equiv q_f$, where $q_f=p/(p-1)$ and $q_f=1$, respectively. A straightforward computation then yields
\begin{equation}
\label{eq:1.5}
    \partial_s v=\Delta v-\frac12\,y\cdot\nabla v-(2-q_f)\frac{|\nabla v|^2}{v}-v+v^2.
\end{equation}
In these model cases, weighted energy estimates yield stationary limits
along sequences of rescaled times. A Pohozaev identity then shows that
the admissible positive limit profiles are constant under the assumptions
of Proposition~\ref{Proposition:1.1}. To obtain convergence to the ODE
value $1$, one must also exclude the zero profile at a blow-up point.

For a general nonlinearity $f$, however, the rescaled equation contains an additional term. More precisely, if
$f'(u)F(u)\to q_f\in[1,\infty)$ as $u\to\infty$, the same change of
variables gives
\begin{equation}
\label{eq:1.6}
    \partial_s v=\Delta v-\frac12\,y\cdot\nabla v-(2-q_f)\frac{|\nabla v|^2}{v}-v+v^2+\frac{|\nabla v|^2}{v}\bigl(f'(u)F(u)-q_f\bigr),
\end{equation}
where $u=u(a+y\sqrt{T-t},t)$.
The last term measures the deviation from the model equation
\eqref{eq:1.5} and produces an error in the corresponding weighted energy
inequality. Controlling the integral of this error in time is essential
to the asymptotic analysis.
In \cite{Ichiya}, this control is obtained from explicit asymptotic
estimates for the nonlinearity. Here we instead impose the following
quantitative condition: there exists $r>1/2$ such that
\begin{equation}
    \label{eq:1.7}
    \limsup_{u\to\infty}\, \bigl| f'(u)F(u)-q_f \bigr| \left(\log\frac{1}{F(u)}\right)^r<\infty.
\end{equation}
{\color{black}The logarithmic correction underlying \eqref{eq:1.7}
was introduced by the first and third authors jointly with Ruf and Terraneo~\cite{FujishimaIokuRufTerraneo}
for $r=1$ and considered by Kumagai~\cite{Kumagai}
for $0<r\le2$.}
Note that, since $1/F(u)\to\infty$ as $u\to\infty$,
the assumption \eqref{eq:1.7} implies $f'(u)F(u)\to q_f$ as $u\to\infty$.
It is also known that $q_f\ge 1$ (see \cite{FujishimaIoku}*{Remark~1.1}).
As shown in Section~5, condition \eqref{eq:1.7} holds for the power and
exponential classes considered in \cites{Chabi,ChabiExponential}, as well
as for more rapidly growing nonlinearities, including
$f(u)=\exp(\exp(u^2))$.

We next address the exclusion of the zero profile. A central contribution of the present paper is a lower bound at interior
blow-up points for general nonlinearities. Building on the nondegeneracy
approach of \cite{GigaKohn1989} and related local boundedness arguments
in \cite{Chabi}, we prove in Theorem~\ref{Theorem:3.1} that
\[
    \liminf_{t\uparrow T}\frac{T-t}{F(u(a,t))}>0
\]
at every interior blow-up point $a$ of a nonnegative type~I solution.
Under the standing assumptions on $f$, this estimate requires only the
existence of a finite limit $q_f=\lim_{u\to\infty}f'(u)F(u)\in[1,\infty)$.
It holds in arbitrary space dimension, without assuming radial symmetry or the
quantitative condition \eqref{eq:1.7}.
This nondegeneracy estimate also supplies the lower bound needed to
control the energy error. Together with the gradient estimate, it gives
a lower bound in a shrinking neighborhood of the blow-up point
(Lemma~\ref{Lemma:4.1}). Condition \eqref{eq:1.7} then bounds the energy
error by the integrable function $Cs^{-2r}$ for sufficiently large $s$,
since $r>1/2$.
Thus the cases $q_f=1$ and $q_f>1$ can be treated by the same argument,
and neither case requires an additional assumption of a pointwise lower
bound at the blow-up point.

\medskip 

Combining these estimates with the classification of stationary limit profiles, we obtain the following main result.

\begin{theorem}
\label{Theorem:1.1}
Let $\Omega\subset\mathbb{R}^N$ be a domain, and let $f\in C^1([0,\infty))$ be nonnegative and satisfy \eqref{eq:1.2} and \eqref{eq:1.7} for some $q_f\ge1$ and $r>1/2$.
Let $u$ be a nonnegative type~{\rm I} blow-up solution of \eqref{eq:1.1} with blow-up time $T\in(0,\infty)$, and let $a\in\Omega$ be a blow-up point of $u$.
Assume further that $N\le2$ if $q_f=1$, and that
$q_f\ge {(N+2)}/{4}$ if $q_f>1$.
Then
\begin{equation}
    \label{eq:1.8}
    \lim_{t\uparrow T}\frac{T-t}{F\bigl(u(a+y\sqrt{T-t},t)\bigr)}=1
\end{equation}
locally uniformly for $y\in\mathbb{R}^N$.
\end{theorem}

For $f(u)=u^p$, the condition $q_f\ge(N+2)/4$ is equivalent to
$(N-2)p\le N+2$, since $q_f=p/(p-1)$. Theorem~\ref{Theorem:1.1}
therefore extends the asymptotic description in
Proposition~\ref{Proposition:1.1} to nonlinearities satisfying
\eqref{eq:1.7}, including examples without exact scale invariance.
The Sobolev critical case is included under the type~I assumption.
The dimensional restrictions enter through the classification of
stationary limit profiles, whereas the nondegeneracy estimate of
Theorem~\ref{Theorem:3.1} holds without these restrictions.
Our arguments are localized near interior blow-up points and require no regularity of $\partial\Omega$, provided that the solution is classical in $\Omega\times(0,T)$.

The paper is organized as follows.
Section~2 establishes local derivative estimates, the weighted energy
inequality \textcolor{black}{associated with the quasi-similarity transformation}, and a Pohozaev identity.
Section~3 then proves the nondegeneracy estimate in
Theorem~\ref{Theorem:3.1}. In Section~4, we use this estimate to control
the energy error and prove Theorem~\ref{Theorem:1.1}.
Finally, Section~5 verifies \eqref{eq:1.7} for the power and exponential classes
considered in \cites{Chabi,ChabiExponential} and gives concrete examples.
It also treats the more rapidly growing nonlinearities
$f(u)=u^q\exp(u^p)$ and $f(u)=\exp(\exp(u^p))$.

\section{Preliminaries}

We collect some preliminary facts that will be used below.
Throughout this paper, $C$ denotes a generic constant whose value may change from line to line.
We write $B_R(a):=\{x\in\mathbb{R}^N:|x-a|<R\}$ and $B_R:=B_R(0)$.

\subsection{Gradient estimates}

In this subsection, we establish local derivative estimates near an interior
blow-up point, with constants that are uniform with respect to the observation point.
This uniformity is a key feature of our estimates and does not follow directly
from the related estimates in \cite{Ichiya}*{Proposition~2.5}.
\begin{proposition}
\label{Proposition:2.1}
Let $\Omega\subset\mathbb{R}^N$ be a domain, and let
$f\in C^1([0,\infty))$ be nonnegative and satisfy \eqref{eq:1.2}.
Assume that the limit
\[
    q_f:=\lim_{u\to\infty}f'(u)F(u)
\]
exists.
Let $u$ be a type~{\rm I} blow-up solution of \eqref{eq:1.1}
with blow-up time $T\in(0,\infty)$.
Then, for every $K\Subset\Omega$ and $t_*\in(0,T)$,
there exists a constant $C>0$ such that
\begin{equation}
\label{eq:2.1}
    \frac{|\nabla u(x,t)|}{f(u(x,t))F(u(x,t))}
    \le\frac{C}{\sqrt{T-t}},\qquad
    \frac{|D^2u(x,t)|+|u_t(x,t)|}{f(u(x,t))F(u(x,t))}
    \le\frac{C}{T-t}
\end{equation}
for all $(x,t)\in K\times[t_*,T)$.
The constant may depend on $u,f,K,t_*,T$ and the type~{\rm I} bound,
but is independent of $(x,t)$.
\end{proposition}

{\color{black}We begin with two preliminary lemmas. Lemma~\ref{Lemma:2.1} establishes
a gradient bound by the maximum principle, as well as a local comparison
estimate for $F(u)$.} Lemma~\ref{Lemma:2.2} then provides interior derivative
estimates. We then 
prove Proposition~\ref{Proposition:2.1},
following the rescaling and interior regularity approach of
\cite{GigaKohn1985}*{Proposition~1}.

\begin{lemma}
  \label{Lemma:2.1}
Let $\Omega\subset\mathbb R^N$ be a domain, and let $f\in C^1([0,\infty))$
be nonnegative and satisfy \eqref{eq:1.2}.
Assume that the limit $q_f$ exists and belongs to $[1,\infty)$.
Let $u$ be a type~{\rm I} blow-up solution of \eqref{eq:1.1}
with blow-up time $T\in(0,\infty)$. 
For every $K\Subset\Omega$, there are $C>0$ and $t_K<T$ such that
\begin{equation}
  \label{eq:2.2}
 \left|\nabla\log\frac1{F(u(x,t))}\right|\le\frac C{\sqrt{T-t}} 
\end{equation}
for $x\in K$ and $t_K<t<T$, 
and
\begin{equation}
  \label{eq:2.3}
 \left|\log\frac{F(u(x,t))}{F(u(\xi,\tau))}\right|\le C 
\end{equation}
for all $|\xi-x|\le2\sqrt{T-t}$, $t_K<t<T$, and $T-2(T-t)\le\tau\le T-(T-t)/4$. 
\end{lemma}

\begin{remark}
The basic idea of the proof is to derive \eqref{eq:2.2} by applying
a maximum principle argument to the logarithmic variable
\[
z(x,t):=\log \frac{1}{F(u(x,t))}.
\]
A direct calculation gives
\begin{equation}
    \notag
\partial_t z-\Delta z
=(f'(u)F(u)-1)|\nabla z|^2+\frac{1}{F(u)}.
\end{equation}
The coefficient of $|\nabla z|^2$ converges to $q_f-1$ as $u\to\infty$.
We wish to exploit this behavior for large $u$, but the maximum principle
argument also involves points where $u$ is not large.
We therefore modify $h_f=fF$ only for bounded values of its argument,
replacing it with a positive function $\widehat h_f\in C^1([0,\infty))$
that agrees with $h_f$ for sufficiently large arguments and whose derivative
is uniformly close to $q_f-1$ on $[0,\infty);$ see \eqref{eq:2.6}.
We then use the modified function
\[
z(x,t):=\int_0^{u(x,t)}\frac{1}{\widehat h_f(s)}\,ds.
\]
For large $u$, this function agrees with $\log(1/F(u))$ up to an additive
constant. Thus, the modification preserves the logarithmic change of
variables where $u$ is large, while providing uniform control of the
quadratic gradient coefficient throughout the maximum principle argument.
After estimating the gradient of the modified function, we transfer the
bound back to $\log(1/F(u))$, as detailed below.
\end{remark}

\begin{proof}
By the type~I assumption, fix $M\ge1$ such that  
\begin{equation}
  \label{eq:2.4}
  \sup_{x\in\Omega}\frac{T-t}{F(u(x,t))}\le M
\end{equation}
for 
$t$ sufficiently close to $T$.
Choose $A\ge1$ such that
\begin{equation}
 \label{eq:2.5}
 A\inf_{\rho\ge0}
 \left\{\frac{\rho}{2(1+\rho)}+\frac{1}{(1+\rho)^2}\right\}>M+2,
\end{equation}
and let $\delta\in(0,1]$ be sufficiently small depending only on $A$ and
$q_f$, as specified in the quadratic gradient estimate below.
For the change of variables, put $h_f(u)=f(u)F(u)$. Then $h_f'(u)=f'(u)F(u)-1\to q_f-1$ as $u\to \infty$, 
and there 
exist $U_0>0$ and a positive $\widehat h_f\in C^1([0,\infty))$ such that
\begin{equation}
  \label{eq:2.6}
 \widehat h_f=h_f \quad \mbox{on} \quad [U_0,\infty),\qquad
 |\widehat h_f'-(q_f-1)|\le\delta \quad \mbox{on} \quad [0,\infty).
\end{equation}
Using this extension, define \[ H(u)=\int_0^u \frac{1}{\widehat h_f(s)}\,ds, 
\qquad z(x,t) = 
H(u(x,t)).\] 
A direct computation gives
\begin{equation}
  \label{eq:2.7}
  \partial_t z-\Delta z=(q_f-1+\theta(x,t))|\nabla z|^2
  +\frac{f(u)}{\widehat h_f(u)},
\end{equation}
where $\theta(x,t):=\widehat h_f'(u(x,t))-(q_f-1)$.
By \eqref{eq:2.4}, \eqref{eq:2.6}, and the identity
$(H+\log F)'=0$ on $[U_0,\infty)$, there is $C_0>0$ such that,
for all $x\in\Omega$ and $t$ sufficiently close to $T$,
\begin{gather}
  \label{eq:2.8}
 |\theta|\le\delta,\qquad
 0\le\frac{f(u)}{\widehat h_f(u)}\le\frac{M+1}{T-t},
 \qquad 
 0\le z\le m(t):=C_0+\log\frac1{T-t}. 
\end{gather}
Indeed, for $u\ge U_0$,
\[
 \frac{f(u)}{\widehat h_f(u)}=\frac1{F(u)}\le\frac M{T-t}, 
 \quad 
 H(u)-H(U_0)
 =\log\frac{F(U_0)}{F(u)}
 \le \log\bigl(MF(U_0)\bigr)+\log\frac1{T-t}.
\]
On $[0,U_0]$, both $H$ and $f/\widehat h_f$ are bounded and
nonnegative, so the same bounds follow by taking $C_0$ sufficiently
large and $t$ sufficiently close to $T$.

With these bounds in hand, we estimate $z(x,t)-z(y,t)$ by applying
{\color{black}a maximum principle argument}.
See \cite{AndrewsClutterbuck2009}*{Section~2} for an exposition of this method,
which goes back to \cite{Kruzhkov1979}.
To apply this argument, write
$K_\varrho:=\{x\in \Omega:\, \operatorname{dist}(x,K)\le\varrho\}$
for $\varrho>0$ and choose $R>0$ such that $K_{4R}\Subset\Omega$.
For any $a\in K_{2R}$, consider the ball $B_R(a)$ and set
\[
 b_0(t)=\frac{m(t)+1}{R^2},\qquad
 \phi(r,t)=A\log\left(1+\frac{kr}{\sqrt{T-t}}\right),
\]
where $k\ge 1$ is a constant. 
{\color{black}Define}
\[
 Z(x,y,t)=z(x,t)-z(y,t)-\phi(|x-y|,t)
                   -b_0(t)(|x-a|^2+|y-a|^2).
\]
It is nonpositive when the two spatial points coincide, and at most $-1$
when either point lies on the boundary of the ball $B_R(a)$.
Consequently, any positive maximum of $Z$ must occur at two distinct
points inside $B_R(a)$. 

To choose the starting time for the maximum principle argument, let
$C_1=C_1(N,A,q_f)>0$ be sufficiently large for the estimates below.
Its choice is independent of $a$, $k$, and the starting time.
For a fixed compact set $K'\Subset\Omega$ and $t_1>0$, 
it follows from the maximum principle that 
\begin{equation}
  \label{eq:2.9}
  d_{K',t_1} := \inf\{u(x,t):\, (x,t)\in K'\times[t_1,T)\}>0.
\end{equation}
Since $b_0=O(|\log(T-t)|)$ as $t\uparrow T$
by \eqref{eq:2.8}, 
choose $t_1\in(0,T)$ sufficiently close to $T$, 
independent of $a,k$,
so that  
\begin{equation}
  \label{eq:2.10}
  (T-t)C_1(b_0+b_0^2R^2)\le 1  
  \quad\mbox{for}\quad t\ge t_1, 
\end{equation} 
and \eqref{eq:2.9} with $K'=K_{2R}$ and \eqref{eq:2.8} hold. 
Furthermore, set
$L_1:=\sup_{K_{3R}}|\nabla z(\cdot,t_1)|$ 
and choose $k\ge1$ with $\phi(2R,t_1)\ge2RL_1$.
The concavity of the logarithm gives
\[ \phi(r,t_1) \ge \left(1-\frac{r}{2R}\right)\phi(0,t_1) + \frac{r}{2R}\phi(2R,t_1) \ge  L_1r 
\quad\mbox{for}\quad r\in [0,2R], \]
yielding $Z(x,y,t_1)\le0$ for $x$, $y\in B_R(a)$.

Suppose, for contradiction, that $Z$ is positive for some pair of points
in $B_R(a)$ at some time $t_2\in (t_1,T)$.
Take the maximum of $Z$ over $\overline{B_R(a)}\times\overline{B_R(a)}\times [t_1,t_2]$, 
and denote a maximum point by $(x,y,t)$.
Then $t_1<t\le t_2$, $x,y\in B_R(a)$, and $x\ne y$.
At this point, 
$\nabla_x Z=\nabla_y Z=0$, 
$Z_t\ge0$, and 
$D_{x,y}^2 Z$ is negative semidefinite.
Write $r=|x-y|$ and $e=(x-y)/r$.
In the following estimates, $\phi_r$, $\phi_{rr}$, and $\phi_t$ are
evaluated at $(r,t)$, and $b_0$ and $b_0'$ at $t$.
{\color{black}Since $\nabla_x Z=\nabla_y Z=0$ at the maximum point, we obtain}
\begin{equation}
  \label{eq:2.11}
  \nabla_x z(x,t)=\phi_r e+2b_0(x-a),\qquad
  \nabla_y z(y,t)=\phi_r e-2b_0(y-a).
\end{equation}
We show that the nonpositivity of $D_{x,y}^2 Z$ gives
\begin{equation}
  \label{eq:2.12}
 \Delta_x z(x,t)-\Delta_y z(y,t)\le4\phi_{rr}+4Nb_0.
\end{equation}
To verify \eqref{eq:2.12}, let $\{e_j\}_{j=1}^{N-1}$ be an
orthonormal basis of the orthogonal complement of $e$, and write
$Z_{vv}=v^\mathsf{T} (D_{x,y}^2 Z)\, v
$.
Define the column vectors
\[
 E=\begin{pmatrix}e\\-e\end{pmatrix},\qquad
 E_j=\begin{pmatrix}e_j\\e_j\end{pmatrix}
 \quad\mbox{for}\quad j=1,\ldots,N-1.
\]
Since $|(x+se)-(y-se)|=r+2s$ for sufficiently small $s$, the
nonpositivity of $D_{x,y}^2Z$ gives
\begin{align*}
 0&\ge E^{\mathsf{T}} (D_{x,y}^2Z) \,E
  =\left.\frac{d^2}{ds^2}Z(x+se,y-se,t)\right|_{s=0}\\
  &=z_{ee}(x,t)-z_{ee}(y,t)-4\phi_{rr}-4b_0.
\end{align*}
On the other hand, $|(x+se_j)-(y+se_j)|=r$, so for each
$j=1,\ldots,N-1$ we have
\begin{align*}
 0&\ge E_j^{\mathsf{T}} (D_{x,y}^2Z) \,E_j
  =\left.\frac{d^2}{ds^2}Z(x+se_j,y+se_j,t)\right|_{s=0}\\
  &=z_{e_je_j}(x,t)-z_{e_je_j}(y,t)-4b_0.
\end{align*}
Summing these $N$ inequalities and using
$\Delta z=z_{ee}+\sum_{j=1}^{N-1}z_{e_je_j}$
yields the claimed bound.
We next estimate the quadratic gradient terms. By \eqref{eq:2.8}
and \eqref{eq:2.11},
\begin{align*}
 (q_f-1+\theta(x,t))|\nabla_x z(x,t)|^2
 -(q_f-1+\theta(y,t))|\nabla_y z(y,t)|^2
 \le C(b_0R\phi_r+b_0^2R^2+\delta\phi_r^2),
\end{align*}
where $C$ depends only on $q_f$; here we have used $\delta\le1$.
Young's inequality and a sufficiently small choice of $\delta$, depending
only on $A$ and $q_f$, therefore give
\begin{equation}
 \label{eq:2.13}
 (q_f-1+\theta(x,t))|\nabla_x z(x,t)|^2
 -(q_f-1+\theta(y,t))|\nabla_y z(y,t)|^2
 \le\frac{\phi_r^2}{A}+C_1b_0^2R^2.
\end{equation}
The constant $C_1$ can be chosen depending only on $N,A,q_f$ and large
enough to absorb the coefficient of $b_0$ in \eqref{eq:2.12} as well.
In particular, these choices can be made before $t_1$ and $k$ are fixed.
Since
\[
 \phi_{rr}=-\frac{\phi_r^2}{A}\le0,
 \qquad 4\phi_{rr}+\frac{\phi_r^2}{A}\le\phi_{rr},
\]
we obtain from $Z_t\ge0$, $b_0'\ge0$, \eqref{eq:2.7},
\eqref{eq:2.8}, \eqref{eq:2.10}, \eqref{eq:2.12}, and \eqref{eq:2.13} that
\[
 \phi_t-\phi_{rr}
 \le C_1(b_0+b_0^2R^2)+\frac{M+1}{T-t}
 \le\frac{M+2}{T-t}.
\]
On the other hand, putting $\rho=kr/\sqrt{T-t}$ and using $k\ge1$, we have
\[
 \phi_t-\phi_{rr}
 =\frac{A}{T-t}
 \left\{\frac{\rho}{2(1+\rho)}+\frac{k^2}{(1+\rho)^2}\right\}
 \ge \frac{A}{T-t}
 \inf_{\rho\ge0}
 \left\{\frac{\rho}{2(1+\rho)}+\frac{1}{(1+\rho)^2}\right\}
 >\frac{M+2}{T-t}
\]
by \eqref{eq:2.5}.
These inequalities yield a contradiction.
Hence $Z(x,y,t)\le0$ for all
$x,y\in\overline{B_R(a)}$ and $t_1\le t<T$.
For each $x\in K_{2R}$, take the center $a=x$. Substituting
$(x,x+\varepsilon \eta)$ and $(x+\varepsilon \eta,x)$ into the inequality
$Z\le0$, dividing by $\varepsilon>0$, and letting
$\varepsilon\downarrow0$, for any unit vector $\eta$, yields 
\begin{equation}
  \label{eq:2.14}
  |\nabla z(x,t)|\le \frac{Ak}{\sqrt{T-t}} 
\end{equation} 
for all $(x,t)\in K_{2R}\times [t_1,T)$. 

We now transfer this gradient bound to the original normalization.
Since $\widehat h_f=h_f$ for large $u$, 
both $\widehat h_f(u(x,t))/h_f(u(x,t))$ and its reciprocal are bounded on
$K_{2R} \times [t_1,T)$ by \eqref{eq:2.9}. Set $v=-\log F(u)$. Then \eqref{eq:2.14} gives
\begin{equation*}
 |\nabla v|=\frac{\widehat h_f(u)}{h_f(u)}|\nabla z|
  \le\frac C{\sqrt{T-t}}.
\end{equation*}
This proves \eqref{eq:2.2}.
To prove the space-time comparison \eqref{eq:2.3}, we first note that the identity
\[ 
\partial_t v-\Delta v=(f'(u)F(u)-1)|\nabla v|^2+\frac1{F(u)}
\] 
together with \eqref{eq:2.4} and \eqref{eq:2.14} yields
\begin{equation}
  \label{eq:2.15}
  |\partial_t v-\Delta v|\le \frac{C}{T-t}
\end{equation}
on $K_{2R} \times [t_1,T)$. 
We then compare different times by spatial averaging. Fix $x\in K$ and take $t$ sufficiently close to $T$ that
$T-2(T-t)\ge t_1$ and $2\sqrt{T-t}<R$.
Choose a fixed $\eta\in C_c^\infty(B_1)$ with $\eta\ge0$ and $\|\eta\|_{L^1(B_1)}=1$,
and set $\eta_\lambda(y)=\lambda^{-N}\eta(y/\lambda)$ for $\lambda>0$.
For $T-2(T-t)\le\sigma\le T-(T-t)/4$, put
\[
 \overline{v}(\sigma)=\int_\Omega \eta_{2\sqrt{T-t}}(y-x)v(y,\sigma)\,dy. 
\]
Here
$(T-t)/4\le T-\sigma\le2(T-t)$, so \eqref{eq:2.2} gives 
\begin{equation}
  \label{eq:2.16}
  \sup_{|y-x|\le2\sqrt{T-t}}|v(y,\sigma)-v(x,\sigma)|
  \le C\frac{2\sqrt{T-t}}{\sqrt{T-\sigma}} \le C,
\end{equation}
yielding 
$|\overline{v}(\sigma)-v(x,\sigma)|\le C$. 
Since
\[ \int_\Omega \left( \Delta\eta_{2\sqrt{T-t}}\right)(y-x)\, dy =0, \]
integration by parts, \eqref{eq:2.15}, and \eqref{eq:2.16} give 
\begin{align*}
 &|\overline{v}'(\sigma)|
 \\ 
 &= \left| 
 \int_\Omega \eta_{2\sqrt{T-t}}(y-x) \left(\partial_t v- \Delta v + \Delta v\right)(y,\sigma)\, dy - 
 \int_\Omega \left( \Delta\eta_{2\sqrt{T-t}}\right)(y-x) v(x,\sigma)\, dy \right| 
 \\
 &\le \frac{C}{T-t} + \int_\Omega \left| \left( \Delta\eta_{2\sqrt{T-t}}\right)(y-x)(v(y,\sigma)-v(x,\sigma))\right|\, dy 
 \le \frac{C}{T-t}. 
\end{align*}
Finally, integrating between $t$ and $\tau$ and using
$|t-\tau|\le C(T-t)$, we obtain
\[
|v(\xi,\tau) - v(x,t)| 
\le |v(\xi,\tau) - \overline{v}(\tau)| + |\overline{v}(\tau)-\overline{v}(t)| + 
|\overline{v}(t) - v(x,t)| \le C 
\]
for all $|\xi-x|\le 2\sqrt{T-t}$ and $T-2(T-t)\le\tau\le T-(T-t)/4$, 
which is \eqref{eq:2.3}.
All required cylinders lie in the fixed interior region 
if $t$ is sufficiently close to $T$.
Choosing $t_K$ so that \eqref{eq:2.2} and \eqref{eq:2.3} hold for every $t>t_K$
completes the proof.
\end{proof}

\begin{lemma}
\label{Lemma:2.2}
Let $Q\subset\mathbb R^N\times\mathbb R$ be a bounded open parabolic
cylinder and let $Q'\Subset Q$ be a compact cylinder.
Let $I\subset\mathbb R$ be a compact interval, and suppose that
$A\in C(I)$ and $g\in C^1(I)$ satisfy
\[
 |I|\le L,\qquad
 \|A\|_{L^\infty(I)}+\|g\|_{L^\infty(I)}
 +\|g'\|_{L^\infty(I)}\le L,
\]
for some $L\ge1$,
where $|I|$ denotes the length of $I$.
If $v\in C^{2,1}(Q)$ with $v(Q)\subset I$ satisfies
\[
 \partial_\tau v-\Delta v=g(v)+A(v)|\nabla v|^2
 \quad\mbox{in}\quad Q,
\]
then
\[
 \sup_{Q'}\bigl(|\nabla v|+|D^2v|+|\partial_\tau v|\bigr)\le C,
\]
where $C$ depends only on $N,L,Q,Q'$.
\end{lemma}

\begin{proof}
If $I$ consists of a single point, the assertion is immediate.
Otherwise, we remove the quadratic gradient term by a change of variables. To this end, fix $v_0\in I$ and define
\[
 \Phi(v)=\int_{v_0}^v
 \exp\left(\int_{v_0}^u A(s)\,ds\right)du.
\]
We may extend $A$ constantly beyond the endpoints of $I$ when defining
$\Phi$. On $I$ we have
\begin{equation}
  \label{eq:2.17}
 e^{-L^2}\le\Phi'\le e^{L^2},\qquad
 |\Phi|+|\Phi''|\le2Le^{L^2},\qquad \Phi''=A\Phi'.
\end{equation}
Setting $w=\Phi(v)$, we obtain
\[
 \partial_\tau w-\Delta w=
 \Phi'(v) (\partial_\tau v-\Delta v) - \Phi''(v)|\nabla v|^2
 =
 \Phi'(v)g(v) 
 =:
 \mathcal G(w). 
\]
Since $\mathcal G(w) = \Phi'(\Phi^{-1}(w))g(\Phi^{-1}(w))$, 
we obtain 
\[
 \mathcal G'(w)=g'(v)+A(v)g(v),  
\]
where $v=\Phi^{-1}(w)$. 
Consequently, the assumptions of the lemma and
\eqref{eq:2.17} imply that $w$ and $\mathcal G(w)$ are uniformly bounded and that
$\mathcal G$ is Lipschitz on $\Phi(I)$ with a constant depending only on $L$.
To apply interior regularity, choose an intermediate cylinder $Q_1$ with $Q'\Subset Q_1\Subset Q$.
The interior H\"older estimate for a bounded right-hand side yields
$\|w\|_{C^{\alpha,\alpha/2}(Q_1)}\le C$ for some $\alpha\in(0,1)$ (see, for example, \cite{LSU1968}*{Chapter~III, Theorem~10.1}).
The interior Schauder estimate,
\cite{LSU1968}*{Chapter~IV, Theorem~10.1}, therefore implies
\[
 \|w\|_{C^{2+\alpha,1+\alpha/2}(Q')}\le C.
\]
Finally, the identities
\[
 \nabla v=\frac{\nabla w}{\Phi'(v)},\qquad
 \partial_\tau v=\frac{\partial_\tau w}{\Phi'(v)},\qquad
 D^2 v=\frac{D^2 w-\Phi''(v)\nabla v\otimes\nabla v}{\Phi'(v)}, 
\]
together with \eqref{eq:2.17} prove the assertion.
\end{proof}

We now combine the local comparison estimate \eqref{eq:2.3}
with the interior regularity estimates of Lemma~\ref{Lemma:2.2}
to prove Proposition~\ref{Proposition:2.1}.

\begin{proof}[Proof of Proposition~\ref{Proposition:2.1}]
Fix $a\in\Omega$, set
$\widetilde R=\min\{\operatorname{dist}(a,\partial\Omega),\sqrt T\}$,
and fix $0<r<\widetilde R$. It is enough to prove the estimates on
$B_r(a)$ at sufficiently late times. 
To this end, for $x\in B_r(a)$ and $t$ sufficiently close to $T$, set
\begin{equation}\label{eq:2.18}
 \lambda=\left(\frac{2(T-t)}{\widetilde R^2}\right)^\frac{1}{2},
 \qquad 
 w(z,\tau)=\frac{\lambda^2}
 {F\bigl(u(x+\lambda z,(1-\lambda^2)T+\lambda^2\tau)\bigr)}.
\end{equation}
If $2(T-t)<(\widetilde R-r)^2$, this is defined on
$B_{\widetilde R}\times[T-\widetilde R^2,T)$.
In the rescaled equations below, $u$ is evaluated at
$(x+\lambda z,(1-\lambda^2)T+\lambda^2\tau)$.
A direct calculation yields 
\begin{equation}\label{eq:2.19}
  \begin{aligned}
    \partial_\tau w-\Delta_z w
    & = w^2+(f'(u)F(u)-2)\frac{|\nabla_z w|^2}{w}
    \\
    &=w^2+(q_f-2)\frac{|\nabla_z w|^2}{w}
    +(f'(u)F(u)-q_f)\frac{|\nabla_z w|^2}{w}.
  \end{aligned}
\end{equation}
By the type~I bound for $u$, there exists $M\ge1$,
independent of $x$ and $t$, such that
\begin{equation}
  \label{eq:2.20}
  0<w(z,\tau)\le \frac{M}{T-\tau}
\end{equation}
for all $(z,\tau)\in
B_{\widetilde R}\times[T-\widetilde R^2,T)$,
provided $t$ is sufficiently close to $T$.

To apply the two preceding lemmas on fixed cylinders, consider the
closed cylinders $Q\Subset Q_+^\circ$ defined by
\begin{align*}
 Q_+=\overline B_{3\widetilde R/4}\times
  \left[T-\frac{7}{8}\widetilde R^2,T-\frac{1}{8}\widetilde R^2\right],\quad 
 Q=\overline B_{\widetilde R/2}\times
  \left[T-\frac{3}{4}\widetilde R^2,T-\frac14 \widetilde R^2\right].
\end{align*}
Put
\[
 \tau_c=T-\widetilde R^2/2,\qquad
 u_c=u(x,t),\qquad w_c=w(0,\tau_c)=\frac{\lambda^2}{F(u_c)},
 \qquad W=\frac{w}{w_c}.
\]
We now show that
\begin{equation}\label{eq:2.21}
 e^{-C}\le W\le e^C\quad\hbox{on }Q_+,
 \qquad 0<w_c\le\frac{2M}{\widetilde R^2}.
\end{equation}
Setting $\xi=x+\lambda z$ and $s=(1-\lambda^2)T+\lambda^2\tau$, we obtain
\[
W=\frac{F(u(x,t))}{F(u(\xi,s))}.
\]
Applying~\eqref{eq:2.3} in Lemma~\ref{Lemma:2.1} to $W$ on $Q_+$,
we obtain $e^{-C}\le W \le e^C$ on $Q_+$.
Furthermore, \eqref{eq:2.20} gives
the latter inequality in \eqref{eq:2.21}.
Dividing \eqref{eq:2.19} by the constant $w_c$ gives
\begin{equation}\label{eq:2.22}
 \partial_\tau W-\Delta W
 =w_cW^2+\frac{f'(u)F(u)-2}{W}|\nabla W|^2.
\end{equation}
To put this equation in the form required by Lemma~\ref{Lemma:2.2}, set
$I:=W(Q_+)$, which is a compact interval containing $1$.
For $W\in I$, define
\[
 U(W)=F^{-1}\left(\frac{F(u_c)}{W}\right),\quad
 g(W) = w_c W^2, \quad 
 A(W)=\frac{f'(U(W))F(U(W))-2}{W}.
\]
For 
$t$ sufficiently close to $T$,
the images of the spatial projections of
$Q$ and $Q_+$ under $z\mapsto x+\lambda z$ lie in a fixed compact
subset of $\Omega$.
Thus, \eqref{eq:2.9} gives $u\ge d>0$ for some $d>0$.
In particular, $U(W)\ge d$ for $W\in I$, so 
\[
 |A(W)|\le(B+2)e^C,\qquad
 |g(W)|+|g'(W)|
 \le\frac{2M}{\widetilde R^2}(e^{2C}+2e^C)
\]
for $W\in I$, 
where $B=\sup\{|f'(u)F(u)|:\, u\ge d\}<\infty$. 
Furthermore, \eqref{eq:2.21} gives $I\subset[e^{-C},e^C]$.
We can therefore apply Lemma~\ref{Lemma:2.2}
to \eqref{eq:2.22} on the open cylinder $Q_+^\circ$ with inner cylinder
$Q$ to obtain
\[
 \sup_{Q}\bigl(|\nabla W|+|D^2W|+|\partial_\tau W|\bigr)\le C,
\]
with $C$ independent of $x$ and $t$.
Since $w=w_cW$ and $W\ge e^{-C}$, we conclude that
\begin{equation}\label{eq:2.23}
 \sup_{Q}
 \frac{|\nabla w|+|D^2w|+|w_\tau|}{w}\le C.
\end{equation}
{\color{black}We now return to the original variables.} At $(z,\tau)=(0,\tau_c)$, we have
\begin{gather*}
 \frac{\nabla_z w}{w}=\lambda\frac{\nabla_xu}{f(u)F(u)},\qquad
 \frac{\partial_\tau w}{w}=\lambda^2\frac{\partial_t u}{f(u)F(u)},\\
 \frac{D_z^2w}{w}=
 \lambda^2\left\{\frac{D_x^2u}{f(u)F(u)}
        +(2-f'(u)F(u))\frac{\nabla_xu\otimes\nabla_xu}{\bigl(f(u)F(u)\bigr)^2}\right\}.
\end{gather*}
It follows from \eqref{eq:2.23} that
\[
 \frac{|\nabla u(x,t)|}{f(u(x,t))F(u(x,t))}\le C\lambda^{-1},\qquad
 \frac{|D^2u(x,t)|+|\partial_t u(x,t)|}{f(u(x,t))F(u(x,t))}\le C\lambda^{-2}.
\]
{\color{black}Together with \eqref{eq:2.18}, these estimates yield the required bounds
on $B_r(a)$ 
for $t$ sufficiently close to $T$.}
Finally, covering any given compact set by finitely many such balls and applying
standard interior regularity on the remaining closed time interval proves \eqref{eq:2.1}.
\end{proof}

\subsection{\textcolor{black}{Quasi-similarity transformation} and energy functional}
\label{Subsection:2.2}
Having established the local derivative estimates, we now introduce the \textcolor{black}{quasi-similarity transformation} and the associated energy functional.
Let $a\in\Omega$ be an interior blow-up point and let $T$ be the blow-up time.
We define \textcolor{black}{the normalized function $v$ and the similarity variables $y,s$} by
\begin{equation}
\label{eq:2.24}
    v(y,s):=\frac{T-t}{F\bigl(u(a+y\sqrt{T-t},t)\bigr)},
    \qquad
    y:=\frac{x-a}{\sqrt{T-t}},\qquad s:=-\log(T-t). 
\end{equation} 
By the strong maximum principle, $u>0$ in $\Omega\times(0,T)$.
Hence $v$ is positive on its domain of definition.
For $b\in\Omega$, let $v_b$ denote the rescaled function in \eqref{eq:2.24}
with $a$ replaced by $b$; in particular, $v=v_a$.
A direct computation shows that $v$ satisfies \eqref{eq:1.6}.

To localize the energy near the blow-up point, let $\alpha\in(0,1)$ and choose a sufficiently large $s_0>\max\{0,-\log T\}$ such that
\[
    \overline{B_{s^\alpha e^{-\frac{s}{2}}}(a)}\subset\Omega
    \qquad\text{for all }s\ge s_0.
\]
We define the energy functional $E[v](s)$, the boundary term $G(s)$, and the remainder term $R(s)$ as follows:
\begin{equation}
\label{eq:2.25}
    \begin{aligned}
        E[v](s) &:=
        \frac{1}{2}\int_{|y|\le s^\alpha}\frac{|\nabla v|^2}{v^{2(2-q_f)}}\rho\, dy
        - \int_{|y|\le s^\alpha}\psi_{f}(v)\rho\, dy,
        \\
        G(s) &:=
        \int_{|y| = s^\alpha}
        \left\{ \frac{\partial_s v}{v^{2(2-q_f)}}\frac{\partial v}{\partial \nu}
        + \frac{\alpha}{s} \left(\frac{|\nabla v|^2}{2v^{2(2-q_f)}}-\psi_f(v)\right)(y\cdot \nu)\right\}
        \rho\, d\sigma,
        \\
        R(s) &:= \frac12
        \int_{|y|\le s^\alpha}\left(\frac{|\nabla v|}{v}\right)^4 v^{2(q_f-1)} \bigl| f'(u)F(u)-q_f \bigr|^2
        \rho\, dy, 
    \end{aligned}
\end{equation}
for $s>s_0$, where $\nu$ is the outward unit normal to $\{|y|=s^\alpha\}$ and
\begin{equation}\label{eq:2.26}
    \rho(y):=\exp\left(-\frac{|y|^2}{4}\right),
    \qquad 
    \psi_f(v):=\begin{cases}
        \displaystyle\frac{v^{2q_f-1}}{2q_f-1}-\frac{v^{2q_f-2}}{2q_f-2} & \mbox{if } q_f>1,
        \\
        v-\log v & \mbox{if } q_f=1. 
    \end{cases}
\end{equation}
With this notation, the energy inequality takes the following form.

\begin{lemma}
    \label{Lemma:2.3}
    Let $v$ be defined by \eqref{eq:2.24}.
    Define $E[v](s)$, $G(s)$, and $R(s)$ by \eqref{eq:2.25}.
    Then
    \begin{equation}
        \label{eq:2.27}
        \begin{aligned}
            \frac{1}{2}\int_{|y|\le s^\alpha}\frac{(\partial_s v)^2}{v^{2(2-q_f)}}\rho\, dy
            & \le
            - \frac{d}{ds}E[v](s) + G(s)
            + R(s) 
        \end{aligned}
    \end{equation}
    for all $s>s_0$, where $u=u(a+y\sqrt{T-t},t)$ 
    {\color{black}in the definition of $R(s)$.}
\end{lemma}

\begin{proof}
    Since the normal velocity of $\{|y|=s^\alpha\}$ is $\alpha s^{\alpha-1}=(\alpha/s)y\cdot\nu$, we have
    \begin{equation}
    \label{eq:2.28}
        \frac{d}{ds}\int_{|y|\le s^\alpha} g(y,s)\,dy
        =\int_{|y|\le s^\alpha} \partial_s g(y,s)\,dy
        +\frac{\alpha}{s}\int_{|y| = s^\alpha} g(y,s)\,(y\cdot\nu)\,d\sigma
    \end{equation}
    for every smooth function $g$ on the moving domain.
    Note also that
    \[
        \nabla\rho=-\frac12y\rho,
        \qquad
        \psi_f'(v)=v^{2q_f-3}(v-1)=\frac{-v+v^2}{v^{2(2-q_f)}}
    \]
    for both $q_f=1$ and $q_f>1$.
    Using these identities, we multiply \eqref{eq:1.6} by $(\partial_s v)v^{-2(2-q_f)}\rho$ and integrate by parts to obtain
    \[
    \begin{aligned}
        \int_{|y|\le s^\alpha}\frac{(\partial_s v)^2}{v^{2(2-q_f)}}\rho\,dy
        & =-\int_{|y|\le s^\alpha}\frac{\nabla v\cdot\nabla(\partial_s v)}{v^{2(2-q_f)}}\rho\,dy
        +(2-q_f)\int_{|y|\le s^\alpha}\frac{|\nabla v|^2\partial_s v}{v^{2(2-q_f)+1}}\rho\,dy
        \\
        &\qquad +\int_{|y|\le s^\alpha}\psi_f'(v)\partial_s v\,\rho\,dy
        +\int_{|y|=s^\alpha}\frac{\partial_s v}{v^{2(2-q_f)}}\frac{\partial v}{\partial\nu}\rho\,d\sigma
        \\
        &\qquad +\int_{|y|\le s^\alpha}\frac{|\nabla v|^2\partial_s v}{v^{2(2-q_f)+1}}\bigl(f'(u)F(u)-q_f\bigr)\rho\,dy
        \\
        & =-\int_{|y|\le s^\alpha}\partial_s\left(\frac{|\nabla v|^2}{2v^{2(2-q_f)}}-\psi_f(v)\right)\rho\,dy
        +\int_{|y|=s^\alpha}\frac{\partial_s v}{v^{2(2-q_f)}}\frac{\partial v}{\partial\nu}\rho\,d\sigma
        \\
        &\qquad 
        +\int_{|y|\le s^\alpha}\frac{|\nabla v|^2\partial_s v}{v^{2(2-q_f)+1}}\bigl(f'(u)F(u)-q_f\bigr)\rho\,dy.
    \end{aligned}
    \]
    Applying \eqref{eq:2.28} to the first integral and using \eqref{eq:2.25}, we obtain
    \[
    % \begin{aligned}
        \int_{|y|\le s^\alpha}\frac{(\partial_s v)^2}{v^{2(2-q_f)}}\rho\,dy
        =-\frac{d}{ds}E[v](s)+G(s)
        +\int_{|y|\le s^\alpha}\frac{|\nabla v|^2\partial_s v}{v^{2(2-q_f)+1}}\bigl(f'(u)F(u)-q_f\bigr)\rho\,dy.
    % \end{aligned}
    \]
    Finally, by Young's inequality and the definition of $R(s)$,
    \[
    \begin{aligned}
        &\int_{|y|\le s^\alpha}\frac{|\nabla v|^2\partial_s v}{v^{2(2-q_f)+1}}\bigl(f'(u)F(u)-q_f\bigr)\rho\,dy
        \\
        &\quad\le\frac12\int_{|y|\le s^\alpha}\frac{(\partial_s v)^2}{v^{2(2-q_f)}}\rho\,dy
        +\frac12\int_{|y|\le s^\alpha}\left(\frac{|\nabla v|}{v}\right)^4v^{2(q_f-1)}\bigl|f'(u)F(u)-q_f\bigr|^2\rho\,dy
        \\
        &\quad=\frac12\int_{|y|\le s^\alpha}\frac{(\partial_s v)^2}{v^{2(2-q_f)}}\rho\,dy+R(s).
    \end{aligned}
    \]
    {\color{black}Substituting this bound into the preceding identity and absorbing the first term on the right-hand side yields \eqref{eq:2.27}.}
\end{proof}

\subsection{Pohozaev identity}

To identify the stationary limits arising from the energy estimate, we next state a Pohozaev identity obtained by combining \cite{GigaKohn1985}*{Proposition~2} and \cite{Liu}*{Lemma~3.1}.

\begin{lemma}
    \label{Lemma:2.4}
    Let $q \ge 1$. Suppose $v \in C^2(\mathbb{R}^N)$ is a bounded positive solution of
    \[
        \Delta v - \frac12\,y\cdot\nabla v - (2-q)\frac{|\nabla v|^2}{v} - v + v^2 = 0
        \quad\mbox{in }\mathbb{R}^N 
    \]
    and, when $q = 1$, assume in addition that $|\nabla v|/v$ is bounded. Then
    \[
        \left(\frac{N(q-1)}{2q-1} + \frac{2-N}{2}\right)\int_{\mathbb{R}^N}\frac{|\nabla v|^2}{v^{2(2-q)}}\rho\,dy
        + 
        \frac12 \left(\frac12 - \frac{q-1}{2q-1}\right)\int_{\mathbb{R}^N}\frac{|\nabla v|^2}{v^{2(2-q)}}|y|^2\rho\,dy
        =0
    \]
    holds.
\end{lemma}

\begin{proof}
    For $q=1$, Lemma~\ref{Lemma:2.4} follows from \cite{Liu}*{Lemma~3.1}.
    {\color{black}For $q>1$, set $\tilde{v}(y):=\bigl((q-1)v(y)\bigr)^{q-1}$.} A direct computation shows that $\tilde{v}$ satisfies
    \[
        \Delta \tilde{v} - \frac{1}{2}y\cdot\nabla \tilde{v} - (q-1)\tilde{v} + \tilde{v}^{\frac{q}{q-1}} = 0
        \quad\mbox{in }\mathbb{R}^N.
    \]
    Setting $p := \frac{q}{q-1} > 1$ and applying \cite{GigaKohn1985}*{Proposition~2}, we obtain
    \[
        \left(\frac{N}{p+1} + \frac{2-N}{2}\right) \int_{\mathbb{R}^N} |\nabla \tilde{v}|^2 \rho\,dy
        + \frac12 \left(\frac12 - \frac{1}{p+1}\right) \int_{\mathbb{R}^N} |y|^2 |\nabla \tilde{v}|^2 \rho\,dy
        = 0. 
    \]
    The identity $\tilde{v}(y) = \bigl(v(y)/(p-1)\bigr)^{\frac{1}{p-1}}$ then gives the conclusion.
\end{proof}

\begin{remark}
    \label{Remark:2.1}
    When $q = 1$ and $N \le 2$, or when $q > 1$ and $q \ge \frac{N+2}{4}$, we have
    \[
        \frac{N(q-1)}{2q-1} + \frac{2-N}{2} \ge 0.
    \]
    Hence, in either case, $|\nabla v| \equiv 0$ in $\mathbb{R}^N$.
\end{remark}

\section{Lower bounds at blow-up points}

In this section, we establish 
the following lower bound for nonnegative type~I
blow-up solutions at interior blow-up points, valid for every finite
$q_f\in[1,\infty)$.

\begin{theorem}\label{Theorem:3.1}
Let $f$ and $u$ satisfy the assumptions of Proposition~$\mathrm{\ref{Proposition:2.1}}$, with any finite
$q_f\in[1,\infty)$. If $a\in\Omega$ is an interior blow-up point, then
\begin{equation}
 \liminf_{t\uparrow T}\frac{T-t}{F(u(a,t))}>0.
 \label{eq:3.1}
\end{equation}
\end{theorem}

The proof combines the nondegeneracy approach of \cite{GigaKohn1989} with
a transformation of the type used in \cite{FujishimaIoku}.
Related local boundedness and nondegeneracy arguments
for nonlinearities without scale invariance can be found in \cite{Chabi}.
{\color{black}Using the gradient estimate of Proposition~\ref{Proposition:2.1}, we develop
an argument based on a weighted $L^2$ estimate.}

{\color{black}We first introduce the notation and record some basic properties used in
the lemma below.} Let $f$ and $u$ satisfy the
assumptions of Proposition~\ref{Proposition:2.1}. Fix $a\in\Omega$ and choose
$r>0$ so that
\[
 5r<\min\{\operatorname{dist}(a,\partial\Omega),\sqrt T\}.
\]
Fix also $t_0\in(0,T)$. Applying \eqref{eq:2.9} with
$K'=\overline{B_{5r}(a)}$ and $t_1=t_0$, we obtain
\[
 d:=\inf\{u(x,t):\,(x,t)\in\overline{B_{5r}(a)}\times[t_0,T)\}>0.
\]
Thus $u$ has a uniform positive lower bound on this entire time interval.
Condition~\eqref{eq:1.2} and the $C^1$ regularity of $f$ imply that
$f(\xi)>0$ for $\xi>0$. Since $f'F$ is continuous on $(0,\infty)$
and tends to the finite limit $q_f$ at infinity, it is bounded above
on $[d,\infty)$. We may therefore fix a finite exponent $\beta$ such that
\[
 \beta\ge1,\qquad \beta+1>\sup_{\xi\ge d}f'(\xi)F(\xi).
\]
This exponent depends only on the fixed data and will be used for all
centers $b\in B_r(a)$ and all starting times considered below.

The function $G_\beta(\xi):=F(\xi)^{-\beta}$ belongs to
$C^2((0,\infty))$, is strictly increasing, and tends to infinity as
$\xi\to\infty$. Direct differentiation yields
\[
 G_\beta'(\xi)=\frac{\beta}{f(\xi)F(\xi)^{\beta+1}},\qquad
 G_\beta''(\xi)=\frac{\beta(\beta+1-f'(\xi)F(\xi))}
                     {f(\xi)^2F(\xi)^{\beta+2}}.
\]
By the choice of $\beta$, we have $G_\beta''(\xi)>0$ for
$\xi\ge d$. In particular, $G_\beta$ is convex on $[d,\infty)$,
which contains all values of $u$ on
$B_{5r}(a)\times[t_0,T)$.

To localize the \textcolor{black}{quasi-similarity transformation} in \eqref{eq:1.4} near $a$,
we apply a spatial cutoff to the $\beta$th power of the rescaled function.
The strong maximum principle gives $u>0$ in $\Omega\times(0,T)$,
so $G_\beta(u(x,t))$ is well defined there.
For each $b\in B_r(a)$, consider the change of variables
\[
 (x,t)=(b+e^{-\frac{s}{2}}y,T-e^{-s}).
\]
For $s>-\log T$ and $x\in\Omega$, define
\[
 v_b(y,s)=\frac{e^{-s}}{F(u(x,t))},\qquad V_b(y,s)=v_b(y,s)^\beta.
\]
Next, fix $\chi\in C_c^\infty(B_{4r}(a))$, with $0\le\chi\le1$ and
$\chi=1$ on $B_{3r}(a)$.
For $(y,s)\in\mathbb R^N\times(-\log T,\infty)$, define
\[
 W_b(y,s)=e^{-\beta s}\chi(x)G_\beta(u(x,t)),
\]
where the product on the right-hand side is extended by zero
outside $\Omega$ in the original variables. The same convention
will be used for products involving derivatives of $\chi$ and for
all whole-space integrals below.
When $x=b+e^{-s/2}y\in\Omega$, this definition gives
\[
 W_b(y,s)=\chi(x)V_b(y,s).
\]
Only the localized products are extended by zero; $v_b$ and $V_b$
are defined only when $x\in\Omega$.
By the type~I bound, there exists a constant $M\ge1$ such that
\[
 \frac{T-t}{F(u(x,t))}\le M
\]
for all $x\in\Omega$ and all 
$t$ sufficiently close to $T$.
It follows immediately that
\begin{equation}
 0\le W_b\le M^\beta
 \label{eq:3.2}
\end{equation}
on $\mathbb R^N$ for all sufficiently large $s$, uniformly in $b\in B_r(a)$.
We take the common lower bound for $s$ sufficiently large that $s\ge0$,
$t>t_0$, the type~I and local gradient bounds hold, and $e^{-s/2}\le2r$.
The last condition ensures that $\chi=1$ on $b+e^{-s/2}B_1$ for every
$b\in B_r(a)$.
Constants may depend on the fixed data, including
$N,u,f,T,a,r,t_0,d,\beta,\chi$ and $M$, but are independent of
$b\in B_r(a)$, the starting time $s_0$, and the current time $s$.

The following lemma shows that smallness of the weighted norm at 
a time sufficiently close to $T$
implies decay at later times.

\begin{lemma}
\label{Lemma:3.1}
With the notation above and the Gaussian weight
$\rho(y)=e^{-|y|^2/4}$ defined in
Section~$\mathrm{\ref{Subsection:2.2}}$, set
\[
 \Phi_b(s)=\int_{\mathbb{R}^N}W_b(y,s)^2\rho(y)\,dy.
\]
There exist $\varepsilon_*>0$, $s_*\ge0$ and $C>0$, independent of
$b\in B_r(a)$, $s_0$ and $s$, such that the conditions
$s_0\ge s_*$ and $\Phi_b(s_0)<\varepsilon_*$ imply
\begin{equation}
 \Phi_b(s)\le C\Phi_b(s_0)e^{-2\beta(s-s_0)},\qquad
 v_b(0,s)\le C\Phi_b(s_0)^{\frac{1}{2\beta}}e^{-(s-s_0)}, 
 \label{eq:3.3}
\end{equation}
for all $s\ge s_0$.
\end{lemma}

\begin{proof}
We first establish the following inequality on
$B_{5r}(a)\times(t_0,T)$:
\begin{equation}
 \partial_t G_\beta(u)-\Delta G_\beta(u)
 \le\beta G_\beta(u)^{1+\frac{1}{\beta}}.
 \label{eq:3.4}
\end{equation}
Indeed, $u\ge d$ on this set, so a direct computation 
{and the convexity of $G_{\beta}$ give}
\[
 \begin{aligned}
 (\partial_t-\Delta)G_\beta(u)
 &=G_\beta'(u)f(u)-G_\beta''(u)|\nabla u|^2\\
 &=\beta G_\beta(u)^{1+\frac{1}{\beta}}-G_\beta''(u)|\nabla u|^2
 \le\beta G_\beta(u)^{1+\frac{1}{\beta}}.
 \end{aligned}
\]
The support of $\chi$ is contained in this ball, so this local
inequality suffices for the localization below.

We next pass to the localized rescaled function. Set
$\mathcal L=\Delta_y-\tfrac12y\cdot\nabla_y$.
Since $V_b=e^{-\beta s}G_\beta(u(x,t))$, the change of variables gives
\[
 \partial_sV_b=-\beta V_b+e^{-(\beta+1)s}\partial_t[G_\beta(u)]
                  -\frac12y\cdot\nabla_yV_b,\qquad 
 \Delta_yV_b=e^{-(\beta+1)s}\Delta_x[G_\beta(u)].
\]
It follows from \eqref{eq:3.4} that
\[
 \partial_sV_b\le\mathcal LV_b-\beta V_b+\beta v_bV_b
\]
whenever $x=b+e^{-s/2}y\in B_{5r}(a)$ and $t>t_0$.
For the cutoff function, direct differentiation yields
\[
 \begin{gathered}
 \partial_s\bigl[\chi(b+e^{-s/2}y)\bigr]
 =-\frac12e^{-s/2}y\cdot\nabla_x\chi(x),\\
 \mathcal L\bigl[\chi(b+e^{-s/2}y)\bigr]
 =e^{-s}\Delta_x\chi(x)-\frac12e^{-s/2}y\cdot\nabla_x\chi(x).
 \end{gathered}
\]
The time derivative therefore cancels the drift term, and hence
\[
 (\partial_s-\mathcal L)\bigl[\chi(b+e^{-s/2}y)\bigr]
 =-e^{-s}\Delta_x\chi(x).
\]
Using $W_b=\chi V_b$ and the product rule, we obtain
\begin{equation}
 \partial_sW_b\le\mathcal LW_b-\beta W_b+\beta v_bW_b
 -2e^{-s/2}\nabla_x\chi\cdot\nabla_yV_b
             -e^{-s}(\Delta_x\chi)V_b.
 \label{eq:3.5}
\end{equation}
{\color{black}Under the zero-extension convention,}
this inequality holds on
$\mathbb R^N$ for all sufficiently large $s$. Each term involving
$v_b$ or $V_b$ is multiplied by $\chi$ or one of its derivatives,
whose supports lie strictly inside $B_{5r}(a)$.

We next derive a weighted $L^2$ estimate.
On every compact time interval under consideration, the local regularity
of $u$ and the $C^2$ regularity of $G_\beta$ imply that $W_b$
has two continuous spatial derivatives and one continuous time derivative.
Its spatial support lies in a fixed compact subset of $\mathbb R^N$
on that interval. Thus we may differentiate $\Phi_b$ under the integral
sign and integrate by parts without boundary terms.
Since $\mathcal L=\rho^{-1}\operatorname{div}(\rho\nabla)$, multiplying
\eqref{eq:3.5} by $2W_b\rho$ and integrating gives
\begin{equation}
 \begin{aligned}
 \Phi_b'&\le-2\int_{\mathbb{R}^N}|\nabla W_b|^2\rho\,dy
       -2\beta\Phi_b
       +2\beta\int_{\mathbb{R}^N}v_bW_b^2\rho\,dy
       +\mathcal R_b(s),
 \end{aligned}
 \label{eq:3.6}
\end{equation}
where
\[
 \mathcal R_b=-4e^{-s/2}\int_{\mathbb{R}^N}
       W_b\nabla_x\chi\cdot\nabla_yV_b\,\rho\,dy
 -2e^{-s}\int_{\mathbb{R}^N}W_bV_b\Delta_x\chi\,\rho\,dy.
\]
Here the derivatives of $\chi$ are evaluated at $x=b+e^{-s/2}y$.

To estimate the nonlinear term and the localization errors, we use
Proposition~\ref{Proposition:2.1} on the fixed compact set
$\overline{B_{4r}(a)}\Subset\Omega$. It gives
\begin{equation}
    \notag
 \left|\nabla_x\log\frac1{F(u(x,t))}\right|
 =\frac{|\nabla_xu(x,t)|}{f(u(x,t))F(u(x,t))}
 \le\frac{C}{\sqrt{T-t}}
\end{equation}
for all $x\in B_{4r}(a)$ and $t\ge t_0$.
Consequently,
\[
 |\nabla_y\log v_b|
 =e^{-s/2}\left|\nabla_x\log\frac1{F(u)}\right|\le C,
 \qquad
 |\nabla_yV_b|=\beta V_b|\nabla_y\log v_b|\le CV_b, 
\]
throughout the region $b+e^{-s/2}y\in B_{4r}(a)$.
The constants are independent of $b,s$ and $y$, even though this region
in the $y$ variables expands as $s$ increases.
For such a point $x=b+e^{-s/2}y$, the segment from $b$ to $x$ lies
in the ball $B_{4r}(a)$. Integrating the logarithmic gradient along
the corresponding segment from $0$ to $y$ gives
\[
 \left|\log\frac{v_b(y,s)}{v_b(0,s)}\right|
 \le\int_0^1|\nabla_y\log v_b(\theta y,s)|\,|y|\,d\theta
 \le C|y|.
\]
Therefore,
\begin{equation}
 e^{-C|y|}v_b(0,s)\le v_b(y,s)\le e^{C|y|}v_b(0,s)
 \label{eq:3.7}
\end{equation}
whenever $b+e^{-s/2}y\in B_{4r}(a)$.

For each fixed $k>0$, the Gaussian weight satisfies
\[
 \int_{\mathbb R^N}e^{kC|y|}\rho(y)\,dy
 \le e^{2k^2C^2}\int_{\mathbb R^N}e^{-|y|^2/8}\,dy<\infty,
\]
where we used $kC|y|\le |y|^2/8+2k^2C^2$.
Since $0\le\chi\le1$ and its support lies in $B_{4r}(a)$,
\eqref{eq:3.7} gives
\[
 \Phi_b(s)
 =\int_{\mathbb R^N}\chi(x)^2v_b(y,s)^{2\beta}\rho(y)\,dy
 \le v_b(0,s)^{2\beta}
       \int_{\mathbb R^N}e^{2\beta C|y|}\rho(y)\,dy
 \le C v_b(0,s)^{2\beta}.
\]
For the lower bound, recall that $\chi=1$ on $b+e^{-s/2}B_1$
for all sufficiently large $s$, uniformly in $b\in B_r(a)$.
Using the lower bound in \eqref{eq:3.7}, we obtain
\[
 \Phi_b(s)\ge\int_{B_1}v_b(y,s)^{2\beta}\rho(y)\,dy
 \ge v_b(0,s)^{2\beta}\int_{B_1}e^{-2\beta C|y|}\rho(y)\,dy
 \ge c v_b(0,s)^{2\beta}.
\]
Thus there exist fixed constants $c,C>0$ such that
\begin{equation}
 c v_b(0,s)^{2\beta}\le\Phi_b(s)\le C v_b(0,s)^{2\beta}
 \label{eq:3.8}
\end{equation}
for all sufficiently large $s$ and all $b\in B_r(a)$.
In particular, 
by \eqref{eq:3.7} and \eqref{eq:3.8},
the nonlinear integral 
{\color{black}on the right-hand side of \eqref{eq:3.6}}
satisfies
\begin{equation}
 \begin{aligned}
 \int_{\mathbb R^N}v_bW_b^2\rho\,dy
 &=\int_{\mathbb R^N}\chi(x)^2v_b^{2\beta+1}\rho\,dy
 \le v_b(0,s)^{2\beta+1}
       \int_{\mathbb R^N}e^{(2\beta+1)C|y|}\rho(y)\,dy\\
 &\le C v_b(0,s)^{2\beta+1}
 \le C\Phi_b(s)^{1+\frac{1}{2\beta}}.
 \end{aligned}
 \label{eq:3.9}
\end{equation}
{\color{black}Moreover, since 
$|\nabla_yV_b|\le CV_b$ throughout the support of the derivatives
of $\chi$ and $0\le W_b\le V_b$ there, the boundedness of}
$\nabla_x\chi$ and $\Delta_x\chi$ gives
\[
 \begin{aligned}
 |\mathcal R_b|
 &\le C(e^{-s/2}+e^{-s})
       \int_{\{y:\,b+e^{-s/2}y\in B_{4r}(a)\}}V_b(y,s)^2\rho(y)\,dy\\
 &\le C(e^{-s/2}+e^{-s})v_b(0,s)^{2\beta}
       \int_{\mathbb R^N}e^{2\beta C|y|}\rho(y)\,dy
 \le Ce^{-s/2}\Phi_b(s).
 \end{aligned}
\]
Here the last step follows from $s\ge0$ and the lower bound in
\eqref{eq:3.8}. 

Combining this estimate with \eqref{eq:3.6} and
\eqref{eq:3.9}, and dropping the nonpositive gradient term, yields
\begin{equation}
\label{eq:3.10}
 \Phi_b'\le\Bigl(-2\beta+C\Phi_b^{\frac1{2\beta}}
                         +Ce^{-\frac{s}{2}}\Bigr)\Phi_b.
\end{equation}
All constants in this inequality are independent of $b$ and $s$.

To ensure decay when the weighted norm is small, choose
$\varepsilon_*>0$ sufficiently small and $s_*\ge0$ sufficiently large
so that all preceding estimates hold for $s\ge s_*$ and
\[
 C\varepsilon_*^{\frac{1}{2\beta}}\le\frac\beta2,
 \qquad Ce^{-\frac{s_*}2}\le\frac\beta2.
\]
Then $\Phi_b'\le-\beta\Phi_b$ as long as
$\Phi_b\le\varepsilon_*$ and $s\ge s_*$.
Let $s_0\ge s_*$ and $\Phi_b(s_0)<\varepsilon_*$.
If there were a first time $s_1>s_0$ with
$\Phi_b(s_1)=\varepsilon_*$, integration up to that time would give
\[
 \Phi_b(s_1)\le\Phi_b(s_0)e^{-\beta(s_1-s_0)}<\varepsilon_*,
\]
a contradiction. By continuity, no such exit is possible, and hence
\[
 \Phi_b(s)\le\Phi_b(s_0)e^{-\beta(s-s_0)}
\]
for all $s\ge s_0$. 

We now substitute this decay estimate into 
the original differential inequality~\eqref{eq:3.10}
to obtain the full decay rate.
Since $W_b=v_b^\beta>0$ on $B_1$, we have $\Phi_b(s)>0$.
We may therefore divide by $\Phi_b$ to obtain
\[
 \frac{d}{ds}\log\Phi_b
 \le-2\beta+C\Phi_b(s_0)^{\frac{1}{2\beta}}e^{-\frac{s-s_0}2}
            +Ce^{-\frac{s}2}.
\]
Integrating over $[s_0,s]$ gives
\[
 \log\frac{\Phi_b(s)}{\Phi_b(s_0)}
 \le-2\beta(s-s_0)
     +2C\Phi_b(s_0)^{\frac{1}{2\beta}}
           \bigl(1-e^{-\frac{s-s_0}2}\bigr)
 +2C\bigl(e^{-\frac{s_0}2}-e^{-\frac{s}2}\bigr).
\]
Consequently,
\[
 \Phi_b(s)
 \le\Phi_b(s_0)
       \exp\!\left(2C\Phi_b(s_0)^{\frac{1}{2\beta}}+2Ce^{-\frac{s_0}2}\right)
       e^{-2\beta(s-s_0)}
 \le C\Phi_b(s_0)e^{-2\beta(s-s_0)}.
\]
The last constant is uniform in $b,s_0$ and $s$, since
$\Phi_b(s_0)<\varepsilon_*$ and $s_0\ge s_*$.
This proves the first inequality in \eqref{eq:3.3}.
Finally, the lower bound in \eqref{eq:3.8} gives
\[
 v_b(0,s)\le C\Phi_b(s)^{\frac{1}{2\beta}}
 \le C\Phi_b(s_0)^{\frac{1}{2\beta}}e^{-(s-s_0)},
\]
which is the second inequality in \eqref{eq:3.3}.
\end{proof}

We now apply Lemma~\ref{Lemma:3.1} to obtain the lower bound at an interior
blow-up point
stated in Theorem~\ref{Theorem:3.1}.

\begin{proof}[Proof of Theorem~\ref{Theorem:3.1}]
{\color{black}Using the construction and notation from Lemma~\ref{Lemma:3.1},
we argue by contradiction.} Suppose that \eqref{eq:3.1} fails. Then there exist
$t_j\uparrow T$ and $s_j=-\log(T-t_j)\to\infty$ such that
$v_a(0,s_j)\to0$. By \eqref{eq:3.8},
\[
 \Phi_a(s_j)\le C v_a(0,s_j)^{2\beta}\to0.
\]
Choose an index $j_*$ so large that, with $s_0=s_{j_*}$,
we have $s_0\ge s_*$ and $\Phi_a(s_0)<\varepsilon_*/2$.
We keep this starting time fixed for the remainder of the proof.

We next extend this smallness to nearby centers. At the fixed time $s_0$,
the zero extension of $\chi G_\beta(u)$ is continuous and compactly supported,
so $W_b(y,s_0)$ depends continuously on $b$ for each $y$.
Since $W_b(y,s_0)^2\rho(y)\le M^{2\beta}\rho(y)$ by \eqref{eq:3.2}
and $\rho\in L^1(\mathbb R^N)$, the dominated convergence theorem shows
that $b\mapsto\Phi_b(s_0)$ is continuous.
Hence there is a fixed $0<\delta<r$ such that
$\Phi_b(s_0)<\varepsilon_*$ for all $b\in B_\delta(a)$.
Applying Lemma~\ref{Lemma:3.1} with the same $\beta$ and the same constants
to all these centers gives
\begin{equation}
    \notag
 \Phi_b(s)\le C\Phi_b(s_0)e^{-2\beta(s-s_0)},\qquad
 v_b(0,s)\le C\Phi_b(s_0)^{\frac{1}{2\beta}}e^{-(s-s_0)}, 
\end{equation}
for all $b\in B_\delta(a)$ and $s\ge s_0$.
Returning to the original variables, we obtain
\begin{equation}
 \frac1{F(u(b,T-e^{-s}))}
 =e^s v_b(0,s)
 \le C e^{s_0}\Phi_b(s_0)^{\frac{1}{2\beta}}
 \le C e^{s_0}\varepsilon_*^{\frac{1}{2\beta}}
 \label{eq:3.11}
\end{equation}
uniformly for $b\in B_\delta(a)$ and $s\ge s_0$.
Since $1/F(\xi)\to\infty$ as $\xi\to\infty$, \eqref{eq:3.11} implies
\[
 \sup_{B_\delta(a)\times[T-e^{-s_0},T)}u<\infty.
\]
{\color{black}Thus $u$ is uniformly bounded in a fixed spatial neighborhood of $a$}
up to time $T$. This contradicts the assumption that $a$ is a blow-up
point and proves \eqref{eq:3.1}.
\end{proof}

\section{Proof of Theorem~\ref{Theorem:1.1}}

We first show how Theorem~\ref{Theorem:3.1} yields a lower bound in a shrinking neighborhood of the blow-up point.
This estimate controls the energy error for all $q_f\ge1$ and extends \cite{Ichiya}*{Lemma~3.2}.

\begin{lemma}
\label{Lemma:4.1}
Let $f$ and $u$ satisfy the assumptions of Proposition~\ref{Proposition:2.1}, and let $a\in\Omega$ be an interior blow-up point of $u$.
Let $\alpha\in(0,1)$ and $\varepsilon\in(0,1)$, and set
\[
    \varphi_{\alpha}(t):=\bigl(-\log(T-t)\bigr)^{\alpha}\sqrt{T-t}.
\]
Then there exists $t_1\in(0,T)$ such that 
\[
    \inf_{|x-a|\le \varphi_{\alpha}(t)}u(x,t)\ge F^{-1}\bigl((T-t)^{1-\varepsilon}\bigr)
\]
for every $t\in(t_1,T)$.
\end{lemma}

\begin{proof}
By Proposition~\ref{Proposition:2.1}, we have
\[
    \biggl|\nabla\log\frac{1}{F(u)}\biggr|\le C(T-t)^{-\frac12}
\]
for all sufficiently small $T-t>0$.
Using Taylor's theorem,
Theorem~\ref{Theorem:3.1}, and the definition of $\varphi_\alpha(t)$, we obtain
\[
    \begin{aligned}
        \log\frac{1}{F\bigl(u(x,t)\bigr)}
        &\ge \log\frac{1}{F\bigl(u(a,t)\bigr)}-C(T-t)^{-\frac12}|x-a|
        \\
        &\ge \log\frac{C}{T-t}-C\Bigl(\log\frac{1}{T-t}\Bigr)^{\alpha}
        \ge \log\frac{1}{(T-t)^{1-\varepsilon}}
    \end{aligned}
\]
hold for all $|x-a|\le \varphi_{\alpha}(t)$ and all sufficiently small $T-t>0$.
This yields the desired estimate.
\end{proof}

We now prove Theorem~\ref{Theorem:1.1}. We first control the boundary and remainder terms in the energy inequality, then show that every limiting profile is stationary and equal to the ODE profile.

\begin{proof}[Proof of Theorem~\ref{Theorem:1.1}]
Let $v$ be defined by \eqref{eq:2.24}.
{\color{black}We establish \eqref{eq:4.1}--\eqref{eq:4.3} below for $|y|\le s^\alpha$ and all sufficiently large $s$.}
{\color{black}Differentiating \eqref{eq:2.24} with respect to $y$, we obtain}
\[    \nabla_y v = \frac{(T-t)^\frac32}{f(u)F(u)^2}\nabla_x u.\]
Since $u$ is a type~{\rm I} blow-up solution, the rescaled function $v$ is bounded. Hence, by Proposition~\ref{Proposition:2.1}, we obtain
\begin{equation}
    \label{eq:4.1}
    \frac{|\nabla_y v|}{v} \le C,
    \qquad
    \frac{|\nabla_y v|}{v^{2-q_f}}
    \le
    Cv^{q_f-1}
    \le C.
\end{equation}
It follows that
\[
    |\nabla_y (v^{q_f-1})| \le C \quad\text{if } q_f>1,
    \qquad
    |\nabla_y (\log v)| \le C \quad\text{if } q_f=1.
\]
Theorem~\ref{Theorem:3.1} and the type~I bound give
$c\le v(0,s)\le C$ for all sufficiently large $s$, with $c>0$.
Integrating the preceding gradient estimates along line segments
from $0$ to $y$, we obtain
\begin{equation}
    \label{eq:4.2}
    |v(y,s)^{q_f-1}| \le C(1+|y|) \quad\text{if } q_f>1,
    \quad
    |\log v(y,s)|\le C(1+|y|)\quad\text{if } q_f=1.
\end{equation}
{\color{black}Combining these estimates with Proposition~\ref{Proposition:2.1} and the identity}
\[
    \partial_s v = -v + \frac{T-t}{f(u)F(u)^2}\left(e^{-s}\partial_t u - \frac{e^{-\frac{s}{2}}}{2} y\cdot \nabla_x u \right),
\]
we obtain
\begin{equation}
    \label{eq:4.3}
    |\psi_f(v)|\le C(1+|y|),
    \qquad
    \frac{|\partial_s v|}{v^{2-q_f}}
    \le
    v^{q_f-1} + Cv^{q_f-1}(1+|y|)
    \le C(1+|y|),
\end{equation}
where $\psi_f$ is defined in \eqref{eq:2.26}.
These estimates allow us to control the boundary term. Indeed, by \eqref{eq:2.25}, we have
\begin{equation}
    \notag
    |G(s)|
    \le
    C (1+s^{\alpha}) s^{\alpha(N-1)} \exp\left(-\frac{s^{2\alpha}}{4}\right)
\end{equation}
for all $s>s_0$, and
\begin{equation}
    \label{eq:4.4}
    \int_{s_0}^\infty |G(s)|\, ds < \infty.
\end{equation}
We also verify the integrability of the remainder term $R(s)$.
{\color{black}Fix $\varepsilon\in(0,1/2]$. Lemma~\ref{Lemma:4.1}, together with \eqref{eq:1.7}, gives}
\[
    \bigl|f'(u)F(u)-q_f\bigr|
    \le C\left(\log\frac{1}{F(u)}\right)^{-r}
    \le C\bigl((1-\varepsilon)s\bigr)^{-r}\le 2^r C s^{-r}
\]
for all $|x-a|\le \varphi_{\alpha}(t)$ and all sufficiently large $s$.
Thus, \eqref{eq:2.25} and \eqref{eq:4.1} yield
\[
    R(s)\le Cs^{-2r}\int_{|y|\le s^\alpha}\rho\,dy\le Cs^{-2r}.
\]
Since $r>\frac12$, we obtain
\begin{equation}
\label{eq:4.5}
    \int_{s_0}^{\infty}R(s)\,ds<\infty.
\end{equation}

By \eqref{eq:4.1} and \eqref{eq:4.3}, the energy satisfies
$|E[v](s)|\le C$ for all sufficiently large $s$.
By \eqref{eq:4.4} and \eqref{eq:4.5}, the integral in
\[
 \mathcal E(s):=E[v](s)
 +\int_s^\infty\bigl(|G(\sigma)|+R(\sigma)\bigr)\,d\sigma
\]
is finite for all sufficiently large $s$.
Lemma~\ref{Lemma:2.3} implies that
\[
 \mathcal E'(s)\le-\frac12\int_{|y|\le s^\alpha}
       \frac{(\partial_s v)^2}{v^{2(2-q_f)}}\rho\,dy\le0.
\]
Since $\mathcal E$ is bounded below, it has a finite limit, and hence,
after increasing $s_0$ if necessary,
\begin{equation}
 \label{eq:4.6}
 \lim_{s\to\infty}E[v](s)=\ell\in\mathbb R,
 \qquad
 \int_{s_0}^\infty\int_{|y|\le s^\alpha}
       \frac{(\partial_s v)^2}{v^{2(2-q_f)}}\rho\,dy\,ds<\infty.
\end{equation}

We now turn to the limiting profiles. Theorem~\ref{Theorem:3.1}
and the first estimate in \eqref{eq:4.1}, integrated along line segments
from $0$ to $y$,
give
\[
 c e^{-C|y|}\le v(y,s)\le C
 \qquad\hbox{for }|y|\le s^\alpha,\quad s\ge s_0,
\]
with $c>0$, after increasing $s_0$ again if necessary.
This provides a uniform positive lower bound on each fixed spatial ball
for all sufficiently large $s$.
Let $(s_j)$ be any increasing sequence with $s_j\to\infty$ and set
\[
 v_j(y,s):=v(y,s+s_j).
\]
For every compact cylinder $Q\Subset\mathbb R^{N+1}$, the functions $v_j$
are defined on a neighborhood of $Q$ for all sufficiently large $j$,
and $c_Q\le v_j\le C$ there, with $c_Q>0$.
{\color{black}The estimates \eqref{eq:4.1} and \eqref{eq:4.3}, and the identity}
\[
 \frac{D_y^2v}{v}
 =\frac{e^{-s}D_x^2u}{f(u)F(u)}
  +(2-f'(u)F(u))\frac{\nabla_yv\otimes\nabla_yv}{v^2},
\]
together with Proposition~\ref{Proposition:2.1}, yield
\[
 \sup_Q\bigl(|\nabla_yv_j|+|D_y^2v_j|+|\partial_sv_j|\bigr)\le C_Q.
\]
Here $f'(u)F(u)$ is bounded on the relevant fixed interior region by
\eqref{eq:2.9} and its finite limit at infinity.
These bounds give equicontinuity of $v_j$ and $\nabla_yv_j$ on compact
cylinders. 
Thus the Arzel\`a--Ascoli theorem and a diagonal argument give a
subsequence such that
\[
 v_j\to v_\infty,\qquad
 \nabla_yv_j\to\nabla_yv_\infty
 \quad\hbox{locally uniformly in }\mathbb R^{N+1}.
\]
After passing to a further subsequence, we also have
$\partial_sv_j\rightharpoonup\partial_sv_\infty$ weakly in
$L^2_{\mathrm{loc}}(\mathbb R^{N+1})$.
The local lower bounds imply $v_\infty>0$.

Writing
\[
 u_j(y,s):=u\bigl(a+e^{-(s+s_j)/2}y,T-e^{-(s+s_j)}\bigr),
\]
we have $F(u_j)=e^{-(s+s_j)}/v_j\to0$ locally uniformly.
Thus, $u_j\to\infty$ and $f'(u_j)F(u_j)\to q_f$ locally
uniformly. The perturbation term in \eqref{eq:1.6} therefore tends to
zero locally uniformly, and the preceding convergences allow us to
pass to the limit in that equation in the sense of distributions.
To prove stationarity, fix $R,S>0$. For sufficiently large $j$,
$B_R\subset B_{(s+s_j)^\alpha}$ whenever $|s|\le S$.
The local upper and lower bounds for $v_j$ and \eqref{eq:4.6} give
\[
 \int_{-S}^S\int_{B_R}|\partial_sv_j|^2\,dy\,ds
 \le C_{R,S}\int_{s_j-S}^{s_j+S}\int_{|y|\le\sigma^\alpha}
       \frac{(\partial_\sigma v)^2}{v^{2(2-q_f)}}\rho\,dy\,d\sigma
 \longrightarrow0.
\]
It follows from the weak convergence of the time derivatives that
$\partial_sv_\infty=0$ in every such cylinder. Hence $v_\infty$ is
independent of $s$ and satisfies
\[
 \Delta v_\infty-\frac12y\cdot\nabla v_\infty
 -(2-q_f)\frac{|\nabla v_\infty|^2}{v_\infty}
 -v_\infty+v_\infty^2=0
 \quad\hbox{in }\mathbb R^N
\]
in the sense of distributions. The bounds for $D_y^2v_j$ imply that
$\nabla v_\infty$ is locally Lipschitz. Since $v_\infty$ is locally
bounded away from zero, elliptic regularity gives
$v_\infty\in C^2(\mathbb R^N)$, so this stationary equation holds
classically.
Finally, we identify the stationary profile. By \eqref{eq:4.1}, the ratio $|\nabla v_\infty|/v_\infty$ is bounded.
Applying Lemma~\ref{Lemma:2.4} under the assumptions
$N\le2$ when $q_f=1$ and $q_f\ge\frac{N+2}{4}$ when $q_f>1$,
we obtain $|\nabla v_\infty|\equiv0$ in $\mathbb{R}^N$.
Since $v_\infty>0$, the stationary equation gives $v_\infty\equiv1$.
{\color{black}Since the sequence $(s_j)$ was arbitrary, this proves \eqref{eq:1.8}.}
\end{proof}
\section{Applications}

We now apply Theorem~\ref{Theorem:1.1} to several classes of nonlinearities.
We begin by verifying \eqref{eq:1.7} for
{\color{black}the power and exponential classes studied by
Chabi~\cites{Chabi,ChabiExponential}, under condition~\eqref{eq:5.1}.}
We then turn
to more rapidly growing examples
that fall outside the scope of Chabi's condition~\eqref{eq:5.1}.

\begin{proposition}
\label{Proposition:5.1}
Let $L\in C^1([u_*,\infty))$ be positive for some $u_*>1$.
Assume that there exist $r>1/2$ and $C>0$ such that
\begin{equation}
    \label{eq:5.1}
    \left|\frac{uL'(u)}{L(u)}\right|\le \frac{C}{(\log u)^r}
    \qquad\text{for all }u\ge u_*.
\end{equation}
Assume further that, for $u\ge u_*$, either $f(u)=u^pL(u)$ with $p>1$
or $f(u)=e^uL(e^u)$.
Then the limit $q_f=\lim_{u\to\infty} f'(u)F(u)$ exists and satisfies
\[
    q_f=\begin{cases}
        \displaystyle\frac{p}{p-1} & \text{if}\quad f(u)=u^pL(u), \\[6pt]
        1 & \text{if}\quad f(u)=e^uL(e^u).
    \end{cases}
\]
Moreover, \eqref{eq:1.7} holds with the above exponent $r$ in both cases.

Conversely, let $f\in C^1([0,\infty))$ be nonnegative and satisfy
\eqref{eq:1.2} and \eqref{eq:1.7} for some $q_f>1$ and $r>1/2$.
Setting $p=q_f/(q_f-1)>1$, the function $L(u):=u^{-p}f(u)$ is positive
and satisfies \eqref{eq:5.1} with the
same exponent $r$, for suitable $u_*>1$ and $C>0$.
\end{proposition}

\begin{proof}
In the first part of the proof, we assume that $f$ satisfies
$f(u)=u^pL(u)$ with $p>1$ or $f(u)=e^uL(e^u)$ for $u\ge u_*$.
{\color{black}We first record a consequence of \eqref{eq:5.1} that will be used in both cases: $L'(u)/L(u)=o(1/u)$, and integration gives}
$\log L(u)=o(\log u)$. Thus, for every $k>0$,
\[
    u^kL(u)=\exp\bigl(k\log u+\log L(u)\bigr)
    =\exp\bigl((k+o(1))\log u\bigr)\to\infty
    \quad\text{as}\quad u\to\infty.
\]
In particular, the integral defining $F(u)$ is finite for all
sufficiently large $u$ in both cases.

We now consider the power case. Suppose that $f(u)=u^pL(u)$ with $p>1$.
Integrating by parts then yields
\begin{equation}
    \notag
    F(u)=\int_u^{\infty}\frac{ds}{s^pL(s)}=\frac{1}{p-1}\frac{u^{1-p}}{L(u)}
    -\frac{1}{p-1}\int_u^\infty\frac{1}{s^pL(s)}\frac{sL'(s)}{L(s)}\,ds.
\end{equation}
Combining this with \eqref{eq:5.1} gives
\begin{equation}
    \label{eq:5.2}
    F(u)=\frac{1}{p-1}\frac{u^{1-p}}{L(u)}\left(
        1+O\bigl((\log u)^{-r}\bigr)
    \right).
\end{equation}
To determine the limiting coefficient, we differentiate $f(u)=u^pL(u)$ to obtain
\[
    f'(u)=p\,u^{p-1}L(u)+u^pL'(u)=u^{p-1}L(u)\left(p+\frac{uL'(u)}{L(u)}\right).
\]
Multiplying this by \eqref{eq:5.2} and using \eqref{eq:5.1}, we obtain
\[
    f'(u)F(u)=\frac{p}{p-1}+O\bigl((\log u)^{-r}\bigr)
    \quad\text{as}\quad u\to\infty,
\]
so $q_f=\frac{p}{p-1}$.
{\color{black}Moreover, \eqref{eq:5.2} gives $\log(1/{F(u)})=O(\log u)$ as $u\to\infty$, so}
\[
    \bigl|f'(u)F(u)-q_f\bigr|\left(\log\frac1{F(u)}\right)^r
    =O(1)
    \quad\text{as}\quad u\to\infty,
\]
again by \eqref{eq:5.1}. This proves \eqref{eq:1.7} in the power case.

{\color{black}Next, suppose that $f(u)=e^uL(e^u)$.}
The change of variables $t=e^s$ and the estimate
\eqref{eq:5.2}, applied with $p=2$, give
\[
    F(u)=\int_u^\infty\frac{e^{-s}}{L(e^s)}\,ds
    =\int_{e^u}^\infty\frac{ds}{s^2L(s)}
    =\frac{e^{-u}}{L(e^u)}\bigl(1+O(u^{-r})\bigr).
\]
Since
\[
    f'(u)=e^uL(e^u)\left(1+\frac{e^uL'(e^u)}{L(e^u)}\right),
\]
assumption \eqref{eq:5.1} yields
\[
    f'(u)F(u)=1+O(u^{-r})
    \quad\text{as }u\to\infty,
\]
which proves $q_f=1$.
Moreover, \eqref{eq:5.1} gives
$(\log L(e^u))'=e^uL'(e^u)/L(e^u)=O(u^{-r})=o(1)$,
so integration yields $\log L(e^u)=o(u)$.
The above formula for $F(u)$ thus implies
$\log(1/F(u))=u+o(u)=O(u)$ as $u\to\infty$, and hence
\[
    \bigl|f'(u)F(u)-1\bigr|\left(\log\frac{1}{F(u)}\right)^r
    = O(1) \quad\mbox{as}\quad u\to\infty. 
\]
{\color{black}Thus, \eqref{eq:1.7} holds with exponent $r$ in the exponential case.}

To prove the converse, suppose that $f$ satisfies \eqref{eq:1.2} and
\eqref{eq:1.7} with $q_f>1$ and $r>1/2$.
Since $(f(u)F(u))'=f'(u)F(u)-1\to q_f-1>0$ as $u\to\infty$,
l'H\^opital's rule gives $f(u)F(u)/u\to q_f-1$. 
Using
\[
    \frac{d}{du}\log\frac1{F(u)}=\frac1{f(u)F(u)},
\]
and integrating, we obtain
\[
    \log\frac1{F(u)}=(1+o(1))\frac{\log u}{q_f-1}
    \quad\text{as}\quad u\to\infty.
\]
Consequently, \eqref{eq:1.7} implies
\[
    f'(u)F(u)=q_f+O\bigl((\log u)^{-r}\bigr),
    \qquad
    (f(u)F(u))'=q_f-1+O\bigl((\log u)^{-r}\bigr).
\]
Since  
\[
\int_{u_0}^u(\log s)^{-r}\,ds=O(u(\log u)^{-r})
\]
for any
fixed $u_0>1$, integration of the second estimate yields
\[
    \frac{f(u)F(u)}u=q_f-1+O\bigl((\log u)^{-r}\bigr).
\]
{\color{black}To obtain the required differential bound, set $p=q_f/(q_f-1)$ and $L(u)=u^{-p}f(u)$. Combining the preceding estimates, we obtain}
\[
    \frac{uL'(u)}{L(u)}
    =\frac{uf'(u)}{f(u)}-p
    =\frac{f'(u)F(u)}{f(u)F(u)/u}-\frac{q_f}{q_f-1}
    =O\bigl((\log u)^{-r}\bigr).
\]
This proves \eqref{eq:5.1} with the same exponent $r$, after choosing
$u_*$ and $C$ appropriately, and completes the proof.
\end{proof}

\begin{remark}
\label{Remark:5.1}
In the exponential case, condition \eqref{eq:1.7} also covers nonlinearities of the form
$f(u)=e^uL(e^u)$ that do not satisfy the differential decay assumption
in \cite{ChabiExponential}. For example, let
\[
    f(u)=\exp\bigl(u+\sqrt{1+u}\bigr),\qquad u\ge0.
\]
Then $L(s)=\exp\bigl(\sqrt{1+\log s}\bigr)$, $s\ge1$, is slowly varying, but
\[
    \frac{sL'(s)}{L(s)}=\frac{1}{2\sqrt{1+\log s}}
\]
is not $O((\log s)^{-\alpha})$ for any $\alpha>1/2$.

To verify \eqref{eq:1.7}, set $g(u)=u+\sqrt{1+u}$.
Since $g'(u)\ge1$, we have $F(u)\le e^{-g(u)}$.
Integration by parts gives
\[
    F(u)=\frac{e^{-g(u)}}{g'(u)}
    +\int_u^\infty e^{-g(s)}\frac{-g''(s)}{g'(s)^2}\,ds.
\]
Using $1\le g'(u)\le3/2$ and
$-g''(u)=1/(4(1+u)^{3/2})$, we obtain
\[
    0\le f'(u)F(u)-1
    \le\frac{g'(u)e^{g(u)}F(u)}{4(1+u)^{3/2}}
    =O( u^{-3/2} ) 
    \quad\mbox{as}\quad u\to\infty. 
\]
Moreover, $e^{-g(u)}/g'(u)\le F(u)\le e^{-g(u)}$ implies
\[
    \log\frac{1}{F(u)}=g(u)+O(1)= u + O(u^{1/2}) 
    \quad\text{as}\quad u\to\infty.
\]
Thus $q_f=1$ and \eqref{eq:1.7} holds with $r=3/2$.
Consequently, Theorem~$\mathrm{\ref{Theorem:1.1}}$ applies to nonnegative type~I
blow-up solutions for this nonlinearity in dimensions $N\le2$.
\end{remark}

Returning to the power case, we recall from Proposition~\ref{Proposition:5.1} that, when $q_f>1$, condition
\eqref{eq:1.7} is equivalent to the differential decay condition 
\eqref{eq:5.1} in
\cite{Chabi} for $f(u)=u^pL(u)$ with $p=q_f/(q_f-1)$.
In particular, our framework includes the following nonlinearities,
for which condition~\eqref{eq:5.1},
or equivalently~\eqref{eq:1.7}, is verified in~\cite{Chabi}:
\begin{itemize}
    \item $f(u)=u^p\bigl(\log(e+u)\bigr)^q$, where $p>1$ and $q\in\mathbb{R}$;
    \item $f(u)=u^p \exp\bigl((\log (u+1))^q \bigr)$, where $p>1$ and $q\in (0,1/2)$; 
    \item $f(u)=u^p\bigl(2+\sin((\log(u+1))^q)\bigr)$, where $p>1$ and $q\in (0,1/2)$; 
    \item $f(u)=u^p (\log (3+u))^{\sin (\log (\log (3+u)))}$, where $p>1$.
\end{itemize}
Similarly, Proposition~\ref{Proposition:5.1} shows that the exponential
nonlinearities in \cite{ChabiExponential} satisfy \eqref{eq:1.7} with $q_f=1$.
Examples include the following:
\begin{itemize}
    \item $f(u)=(u+K)^q e^u$, where $K>0$ and $q\in\mathbb{R}$,
    or $K=0$ and $q\ge1$;
    \item $f(u)=\bigl(\log(u+K)\bigr)^q e^u$, where $K>1$ and
    $q\in\mathbb{R}$, or $K=1$ and $q\ge1$;
    \item $f(u)=\exp\bigl(u\pm(u+1)^\nu\bigr)$, where
    $\nu\in(0,1/2)$;
    \item $f(u)=(u+1)^{\sin(\log(u+1))}e^u$.
\end{itemize}
Beyond this exponential class, our framework includes the following
rapidly growing nonlinearities with $q_f=1$:
\begin{itemize}
    \item $f(u)=u^q\exp\bigl(u^p\bigr)$, where $p>1$ and $q\in \{0\}\cup [1,\infty)$;
    \item $f(u)=\exp\bigl(\exp(u^p)\bigr)$, where $p\ge1$.
\end{itemize}
Thus, we can also treat examples such as $f(u)=\exp\bigl(\exp(u^2)\bigr)$, which could not be handled in \cite{Ichiya} for technical reasons.
{\color{black}In the next two subsections, we verify \eqref{eq:1.7} for these rapidly growing nonlinearities, both of which satisfy $q_f=1$.} 
Consequently, in dimensions $N\le2$, Theorem~\ref{Theorem:1.1} applies to type~I blow-up solutions at the blow-up point, 
without imposing radial symmetry on the solutions.
{\color{black}Further examples of exponential-type nonlinearities satisfying \eqref{eq:1.7} can be found in \cite{FujishimaIokuRufTerraneo}*{Section~4}.}

\subsection{The case $f(u)=u^q\exp(u^p)$}

We first consider $f(u)=u^q\exp\bigl(u^p\bigr)$ with $p>1$ and $q\in \{0\}\cup [1,\infty)$.
This nonlinearity was studied in \cite{Ichiya}.
To check the quantitative condition, we start from
\begin{equation}
    \notag
    F(u)=\int_u^{\infty}\frac{ds}{f(s)}=\int_u^{\infty} s^{-q}\exp\bigl(-s^p\bigr)\,ds.
\end{equation}
By the change of variables $r=s^p$, we obtain
\[
    F(u)=\frac1p\int_{u^p}^{\infty} r^{\frac{1-q}{p}-1}e^{-r}\,dr.
\]
Integrating by parts twice and using the fact that $r^\gamma e^{-r}\to0$ as $r\to\infty$ for every $\gamma\in\mathbb{R}$, we obtain
\[
    F(u)
    =\frac1p\,u^{1-q-p}e^{-u^p}
    \left(1-\frac{q+p-1}{p\,u^p}\right)
    +\frac{(q+p-1)(q+2p-1)}{p^3}
    \int_{u^p}^{\infty}r^{\frac{1-q}{p}-3}e^{-r}\,dr.
\]
To estimate the last integral, note that, for every $\gamma\in\mathbb{R}$ and $v\ge1$,
\[
\begin{aligned}
    \int_v^\infty r^\gamma e^{-r}\,dr
    &=v^\gamma e^{-v}\int_0^\infty\left(1+\frac{\sigma}{v}\right)^\gamma e^{-\sigma}\,d\sigma
    \\
    &\le v^\gamma e^{-v}\int_0^\infty(1+\sigma)^{\max\{\gamma,0\}}e^{-\sigma}\,d\sigma
    \le C_\gamma v^\gamma e^{-v},
\end{aligned}
\]
where $C_\gamma>0$ depends only on $\gamma$.
Taking $v=u^p$ and $\gamma=(1-q)/p-3$, we see that the remainder is
$O(u^{1-q-3p}e^{-u^p})$.
Factoring out $p^{-1}u^{1-q-p}e^{-u^p}$ therefore yields
\begin{equation}
    \label{eq:5.3}
    F(u)=\frac{1}{p}\,u^{1-q-p}e^{-u^p}\left(1-\frac{q+p-1}{p\,u^p}+O\left(\frac{1}{u^{2p}}\right)\right)
    \qquad\text{as }u\to\infty.
\end{equation}
Next, we compute the derivative:
\[
    f'(u)=q\,u^{q-1}e^{u^p}+p\,u^{q+p-1}e^{u^p}=u^{q-1}e^{u^p}\bigl(q+p u^p\bigr).
\]
Multiplying this by \eqref{eq:5.3}, we obtain
\[
    f'(u)F(u)=1-\frac{p-1}{p}\,\frac{1}{u^p}+O\left(\frac{1}{u^{2p}}\right)
    \qquad\text{as }u\to\infty.
\]
In particular, $q_f=\lim_{u\to\infty} f'(u)F(u)=1$ and
\begin{equation}
    \notag
    \bigl|f'(u)F(u)-q_f\bigr|=O\left(\frac{1}{u^p}\right)
    \qquad\text{as }u\to\infty.
\end{equation}
{\color{black}To complete the verification of \eqref{eq:1.7}, we use \eqref{eq:5.3} to obtain}
\[
    \log\frac{1}{F(u)}=u^p+(q+p-1)\log u+O(1)
    \qquad\text{as }u\to\infty.
\]
Therefore,
\[
    \bigl|f'(u)F(u)-q_f\bigr|\,\log\frac{1}{F(u)}=O(1)
    \qquad\text{as }u\to\infty,
\]
which proves \eqref{eq:1.7} for $f(u)=u^q\exp(u^p)$.

\subsection{The case $f(u)=\exp(\exp(u^p))$}

We next consider $f(u)=\exp\bigl(\exp(u^p)\bigr)$ with $p\ge1$. As in the preceding case, we begin by estimating the integral defining the normalization:
\[
    F(u)=\int_u^{\infty}\frac{ds}{f(s)}=\int_u^{\infty}\exp\bigl(-\exp(s^p)\bigr)\,ds.
\]
{\color{black}With the change of variables $r=\exp(s^p)$, we obtain}
\[
    F(u)=\frac1p\int_{\exp(u^p)}^{\infty}(\log r)^{\frac1p-1}\,\frac{e^{-r}}{r}\,dr.
\]
Since
\[
    \frac{d}{dr}\left(\frac{(\log r)^{\frac1p-1}}{r}\right)
    =\frac{(\log r)^{\frac1p-1}}{r^2}
    \left(\frac{\frac1p-1}{\log r}-1\right),
\]
integration by parts gives
\[
\begin{aligned}
    F(u)
    =\frac1p\,u^{1-p}e^{-u^p}e^{-\exp(u^p)}
    +\frac1p\int_{\exp(u^p)}^\infty
    \frac{(\log r)^{\frac1p-1}}{r^2}
    \left(\frac{\frac1p-1}{\log r}-1\right)e^{-r}\,dr.
\end{aligned}
\]
For $u\ge1$, we estimate the remaining integral as follows:
\[
\begin{aligned}
    \left|F(u)-\frac1p\,u^{1-p}e^{-u^p}e^{-\exp(u^p)}\right|
    \le\frac{C}{p}e^{-u^p}\int_{\exp(u^p)}^\infty
    \frac{(\log r)^{\frac1p-1}}{r}e^{-r}\,dr
    =Ce^{-u^p}F(u).
\end{aligned}
\]
{\color{black}Dividing by $F(u)>0$ and taking reciprocals, we obtain}
\begin{equation}
    \label{eq:5.4}
    F(u)=\frac{1}{p}\,u^{1-p}\exp\bigl(-u^p\bigr)\exp\bigl(-\exp(u^p)\bigr)\left(1+O\bigl(e^{-u^p}\bigr)\right)
    \qquad\text{as }u\to\infty.
\end{equation}
The derivative is given by
\[
    f'(u)=f(u)\,\frac{d}{du}\bigl(\exp(u^p)\bigr)=p\,u^{p-1}\exp(u^p)\,f(u).
\]
Multiplying this by \eqref{eq:5.4}, we obtain
\[
    f'(u)F(u)=1+O\bigl(e^{-u^p}\bigr)
    \qquad\text{as }u\to\infty.
\]
In particular, $q_f=\lim_{u\to\infty}f'(u)F(u)=1$ and
\begin{equation}
    \notag
    \bigl|f'(u)F(u)-q_f\bigr|=O\bigl(e^{-u^p}\bigr)
    \qquad\text{as }u\to\infty.
\end{equation}
Finally, to compare this error with the logarithmic factor in \eqref{eq:1.7}, we use \eqref{eq:5.4} to obtain
\[
    \log\frac{1}{F(u)}=\exp(u^p)+u^p+(p-1)\log u+O(1)
    \qquad\text{as }u\to\infty.
\]
Therefore,
\[
    \bigl|f'(u)F(u)-q_f\bigr|\,\log\frac{1}{F(u)}=O(1)
    \qquad\text{as }u\to\infty,
\]
which proves \eqref{eq:1.7} for $f(u)=\exp(\exp(u^p))$.

\medskip 

\noindent 
\textbf{Use of AI.}
{\color{black}During the preparation of this manuscript, the authors used ChatGPT and Codex (OpenAI),
with the GPT-6 Astra and GPT-6.1 Sol models, to assist with language editing
and with the development and verification of mathematical arguments.}
The authors reviewed and verified all content, including
the mathematical statements and proofs, and take full
responsibility for the final manuscript.

\medskip 

\noindent 
\textbf{Acknowledgements.} 
{\color{black}This work was partially supported by JSPS KAKENHI (grant numbers: 23K03179 and 26K00610).}

\end{document}